\documentclass[pdflatex,10pt]{article}
\usepackage[utf8]{inputenc}

\usepackage{amsmath,color}
\usepackage{amssymb,amsthm}
\usepackage{mathtools}
\usepackage[permil]{overpic}
\usepackage[multiple]{footmisc}
\usepackage{xcolor}
\usepackage{pict2e} 
\usepackage{xr}
\usepackage[left=1.6cm,right=1.6cm,top=2.50cm,bottom=2.50cm]{geometry}
\usepackage{graphicx}
\usepackage[font=small,labelfont=bf,
   justification=justified,
   format=plain,labelsep=space]{caption}
\usepackage{subcaption}
\usepackage{float}
\usepackage[english]{babel}
\usepackage[font=small,labelfont=bf,justification=centering]{caption}
\usepackage[T1]{fontenc}
\usepackage{lastpage}
\usepackage{enumerate}
\usepackage{enumitem}
\usepackage{lmodern}
\usepackage{array}
\usepackage{bm}
\usepackage{multirow}
\usepackage{dsfont}
\usepackage{tensor}
\usepackage{fancyhdr}
\usepackage{listings}
\usepackage{dsfont}
\usepackage{siunitx}
\usepackage{comment}
\usepackage{titling}
\usepackage{lipsum}
\usepackage{tabularx}
\usepackage{verbatim}
\usepackage{authblk}
\usepackage{csquotes}
\usepackage{bbm} 
\usepackage{hyperref} 
\usepackage{ctable} 
\usepackage{multirow} 
\usepackage{titlesec} 
\usepackage{marvosym}
\usepackage{upgreek} 
\usepackage{relsize,exscale} 
\usepackage[
backend=biber,
style=numeric,
giveninits=true,
sorting=anyt
]{biblatex}

\allowdisplaybreaks

\graphicspath{{Fig/}}

\addto\captionsenglish{}

\makeatletter
\newcommand{\thickhline}{%
    \noalign {\ifnum 0=`}\fi \hrule height 1pt
    \futurelet \reserved@a \@xhline
}
\newcolumntype{"}{@{\hskip\tabcolsep\vrule width 1pt\hskip\tabcolsep}}
\makeatother

\usepackage{adjustbox}

\mathtoolsset{showonlyrefs}

\makeatletter
\newsavebox{\@brx}
\newcommand{\llangle}[1][]{\savebox{\@brx}{\(\m@th{#1\langle}\)}%
  \mathopen{\copy\@brx\kern-0.5\wd\@brx\usebox{\@brx}}}
\newcommand{\rrangle}[1][]{\savebox{\@brx}{\(\m@th{#1\rangle}\)}%
  \mathclose{\copy\@brx\kern-0.5\wd\@brx\usebox{\@brx}}}
\makeatother

\newtheorem{deff}{Definition}[section]
\newtheorem{prop}[deff]{Proposition}
\newtheorem{thm}[deff]{Theorem}
\newtheorem{lm}[deff]{Lemma}
\newtheorem{cor}[deff]{Corollary}
\theoremstyle{definition}
\newtheorem{example}[deff]{Example}
\newtheorem{remark}[deff]{Remark}

\title{Warning Signs for Boundary Noise\\ and their Application to an Ocean Boussinesq Model}

\author[1,$\dag$]{P. Bernuzzi}
\author[2]{H. A. Dijkstra}
\author[1,3]{C. Kuehn}
\affil[1]{\footnotesize{Technical University of Munich, School of Computation Information and Technology, Department of
Mathematics, Boltzmannstraße 3, 85748 Garching, Germany}}
\affil[2]{\footnotesize{Institute for Marine and Atmospheric research Utrecht, Department of Physics Utrecht University, Princetonplein 5, 3584 CC Utrecht, The Netherlands}}
\affil[3]{\footnotesize{Technical University of Munich, Munich Data Science Institute, Walther-von-Dyck-Straße 10, 85748 Garching, Germany}}
\affil[$\dag$]{Author to whom any correspondence should be addressed. Email: paolo.bernuzzi@ma.tum.de}

\date{\today}

\begin{document}

\maketitle{}

\begin{abstract}
In this paper, we construct and discuss early-warning signs of the approach of a parameter to a deterministic bifurcation on a stochastic partial differential equation (SPDE) model with Gaussian white-noise on the boundary. We specifically focus on the qualitative behaviour of the time-asymptotic autocovariance and autocorrelation of the solutions of the linearized system. We also discuss the reliability of the tools from an analytic perspective and through various examples. Among those, the application of the early-warning signs for an ocean Boussinesq model is explored through numerical simulations. The analytic results obtained expand on previous work and present valuable early-warning signs for various applications.
\end{abstract}

\small{\textbf{Keywords:} Early-warning signs, SPDEs, scaling law, boundary noise, atlantic meridional overturning circulation, Boussinesq.}

\small{\textbf{MSC codes:} 35R60, 37H20, 60H15.}

\small{\textbf{Acknowledgments:} This project has received funding from the European Union’s Horizon 2020 research and innovation programme under Grant Agreement 956170. The authors are thankful to Sven Baars, Gianmarco Del Sarto and Jelle Soons for their appreciated help.}

\pagestyle{fancy}
\fancyhead{}
\renewcommand{\headrulewidth}{0pt}
\fancyhead[C]{\textit{Warning Signs for Boundary Noise\\ and their Application to an Ocean Boussinesq Model}}

\section{Introduction}

The prediction of sudden events holds great importance in climate science due to the impact of the studied phenomena. The dynamics that describe various events can be modeled through fast-slow stochastic partial differential equations (SPDEs) perturbed by boundary noise. Such choice can be justified in applications as the weak effect of an external component in the model. Much analytic theory for SPDEs with boundary noise has already been developed, with a focus on the existence and structure of the solutions \cite{da1993evolution,da1996ergodicity,sowers1994multidimensional}. Yet, the study of observables that predict sudden qualitative changes induced by crossing a bifurcation threshold has not been considered and we aim to contribute in closing this significant gap.\medskip

The employment of asymptotic variance and autocorrelation as warning signals in climate science is an example of the use of such tools. The assumption of the ergodicity of the system can justify the use of such abstract signals in real-life applications and can be found on data sets associated with the end of the Younger Dryas in the tropical Atlantic \cite{lenton2012early}, the collapse of the Atlantic Meridional Overturning Circulation \cite{boers2021observation,ditlevsen2023warning}, the melting of
the Western Greenland Ice Sheet \cite{boers2021critical} and the impact of habitat destruction on different species \cite{drake2010early}. The corresponding models display abrupt transitions, or hysteresis, which underline the importance of the reliability of the signals. \medskip

The study of fast-slow differential equations with critical transitions is a well-known field in mathematics~\cite{kuehn2015multiple}. The typical dynamics of multiple time scale systems display slow motion for long times interspersed by fast abrupt changes \cite{KU4,KU}. Such behaviours are often associated with the crossing of a bifurcation point of the fast subsystem. The sudden nature of such an event justifies the interest in constructing observables that predict critical transitions. Such signals are referred to as early-warning signs \cite{KU2,kuehn2015early}. An important tool for constructing warning signs from data is critical slowing down~\cite{Wiesenfeld1}, induced by decreased dissipation in the drift term. Such events can be observed in real-life data through small perturbations induced by stochastic components \cite{hasselman1976stochastic}. Inspection of observables associated with such term can present valuable tools to predict critical transitions \cite{dakos2008slowing,ditlevsen2010tipping,kubo1966fluctuation,wissel1984universal}. Following \cite{bernuzzi2023bifurcations,ditlevsen2010tipping,KU2,KU}, we discuss the properties of the time-asymptotic autocovariance and autocorrelation of the solution of the linearized model along different directions in the space of square-integrable space.\medskip

More specifically, we consider differential equations of the form
\begin{align} \label{eq:fastslow}
\left\{\begin{alignedat}{2}
    &\text{d}u(x,t)=\left(F_1(p(t))\; u(x,t)+F_2(u(x,t),p(t))\right)\, \text{d}t &&\qquad \text{for $x\in\mathcal{X}$}, \\ 
    &\Gamma(u (x,t),p(t)) = \sigma_0 B \dot{W} (x,t) &&\qquad \text{for $x\in\partial\mathcal{X}$},\\
    &\text{d}p(t)= \epsilon G(u(\cdot,t),p(t))\, \text{d}t &&\qquad,
\end{alignedat}\right.
\end{align}
with $(x,t) \in \mathcal{X} \times [0, \infty)$ for a connected set $\mathcal{X}\subset\mathbb{R}^N$ with nonempty interior and $N\in\mathbb{N}_{>0}:=\left\{1,2,3,\ldots\right\}$. We suppose that $F_1(p(t))$ is a linear operator for any $t$ and that $F_2$ and $G$ are sufficiently smooth nonlinear maps. We consider $\sigma_0>0$, $0<\epsilon\ll 1$ and with the notation $\dot{W}$ we refer to a Gaussian white-noise process on the boundary of $\mathcal{X}$, whose precise properties are described further below. The operator $\Gamma$ defines the nature of the boundary conditions and $B$ the effect of the noise on such region of space. The solutions $u:\mathcal{X} \times [0, \infty)\to \mathbb{R}$ and $p: [0, \infty)\to \mathbb{R}$ exist almost surely in the Sobolev space of degree $\alpha$, $\mathcal{H}^\alpha(\mathcal{X})$. From now, we denote $u(x,t)$ by $u$ and $p(t)$ as $p$. \medskip

We assume that, for any $p$, there exists the steady state $u_{\ast}$ of the deterministic partial differential equation (PDE) corresponding to \eqref{eq:fastslow} with $\sigma_0=0$. In order to determine the deterministic local stability of $u_{\ast}$, we study the linear operator
\begin{align*}
    A(p) := F_1(p) + \text{D}_{u_{\ast}} F_2(u_{\ast},p) \, ,
\end{align*}
associated to the linear boundary operator $\gamma(p)$, for $\text{D}_u$ that indicates the Fréchet derivative of $u$ on a suitable Banach space described in detail below.\medskip

Suppose that for $p <\lambda$ the spectrum of $A(p)$ restricted on the space $\left\{f\in L^2(\mathcal{X}) \;\big| \;\gamma(p) f \equiv 0\right\}$ is contained in $\left\{ z: \text{Re}(z) < 0\right\}$, whereas for $p=\lambda$ the spectrum has elements in $\left\{ z: \text{Re}(z) = 0\right\}$. Therefore, the fast subsystem, observed under the assumption $\epsilon=0$, has a bifurcation point at $p=\lambda$. In the fast-slow system~\eqref{eq:fastslow}, with $\epsilon>0$, the slow dynamics $\partial_t p = \epsilon G(0,p)$ can drive the solution of the system to the bifurcation point at $p=\lambda$ and induce a critical transition. Setting $\epsilon=0$, the variable $p$ becomes a parameter on which the motion of $u$ depends. Moreover, shifting $p$, it can cross the bifurcation threshold $p=\lambda$ associated to the fast SPDE dynamics described by the first two equations in \eqref{eq:fastslow}.\medskip

For the rest of this work, we consider $\sigma_0=1$, for simplicity, the case $\epsilon = 0$ and assume that $p$ is a parameter. This is the natural starting point as has been shown for many other, less complex, questions for early-warning signs in stochastic systems~\cite{berglund2006noise,kuehn2013mathematical}. Particular focus in this work is on the linearized leading-order fast system
\begin{equation} \label{eq:fast_linear}
\left\{\begin{alignedat}{2}
    &\text{d} u(x,t)= A(p) \; u(x,t) \text{d}t &&\qquad \text{for $x\in\mathcal{X}$},\\
    &\gamma(p)\; u (x,t)= B \dot{W} (x,t) &&\qquad \text{for $x\in\partial\mathcal{X}$} \;.
\end{alignedat}\right.
\end{equation}
We consider the initial conditions of \eqref{eq:fast_linear}, labeled $u_0$, close to the null function and the initial conditions of \eqref{eq:fastslow} equal to $u_0+u_\ast$. The solutions of \eqref{eq:fastslow} and \eqref{eq:fast_linear} present similar qualitative behaviour for large times if $p<\lambda$ and $A(p)$ restricted on $\left\{f\in L^2(\mathcal{X}) \;\big| \;\gamma(p) f \equiv 0\right\}$ has eigenvalues with real part distant from $0$. In particular, we shall always assume the dissipativity of $A(p)$ for $p<\lambda$. Then the linear term of $F_1(p)+F_2(u_\ast,p)$, constrains the solution in a neighbourhood of the initial condition for long times. Therefore, early-warning signs, observing the solution of each system, predict the approach of $p$ to $\lambda$ in a similar manner, up to proximity to the bifurcation threshold~\cite{bernuzzi2023bifurcations,KU2,horsthemke1984noise,kuehn2013mathematical}.\medskip

The introduction of boundary noise extends the use of the early-warning signs to models perturbed stochastically by external forces on the surface of the studied domain. Such effects are associated with wind and salinity-fluxes in ocean models \cite{baars2017continuation,hasselman1976stochastic}. They have also been studied in the heat equation, as additive noise, and Burgers equation, as multiplicative noise \cite{munteanu2019boundary}. The assumption of boundary noise presents several technical difficulties compared to the space-time white noise respective \cite{KU2}. As a result, the time-asymptotic autocovariance is observed in a subset of the square-integrable function space, yet large. The construction of the boundary condition has a twofold role: it affects the restrictions on the assumptions on the noise to ensure existence of the solution of the system, such as in the case of the heat equation with noise on Neumann boundary conditions on the interval \cite{da1993evolution,da1996ergodicity}; the interplay between the linear operator $A$ associated to the drift component, the operator $\gamma$ that defines the linearized boundary conditions, $B$ that indicates the effect of the noise and their dependence on $p$ can affect or silence the variance as an early-warning sign, in the sense that it can change rate of divergence of the observable as $p$ approaches $\lambda$. The silencing of such early-warning signs can also be observed under the assumption of white noise in space and time, yet the role of $\gamma$ and $B$ is not as relevant as in the case discussed \cite{bernuzzi2023bifurcations,KU2}.
\medskip

The paper is structured as follows: Section \ref{sec:Preliminaries} describes the assumptions of our study. The analytic construction of the early-warning signs is presented in Section \ref{sec:EWS}. Examples of their application to specific models, along with numerical cross-validation of the results, are discussed in Section \ref{sec:examples}. 

\section{Preliminaries} \label{sec:Preliminaries}

Consider \eqref{eq:fast_linear} for $u(\cdot,0)=u_0(\cdot)$ and $\lambda_0\leq p\leq\lambda$. We label as $\mathcal{D}$ the domain of an operator; we also denote as $\mathcal{H}:=L^2(\mathcal{X})$ the complex Hilbert space of square-integrable functions with domain $\mathcal{X}$ and whose inner product is labeled as $\left\langle\cdot,\cdot\right\rangle$. We define the linear operator, dependent on the parameter $p$,
\begin{equation*}
    A(p):\mathcal{D}(A(p))\subseteq \mathcal{H} \mapsto \mathcal{H} \;.
\end{equation*} 
We introduce the cylindrical Wiener process following the definition in \cite{DP}. We set the time ${t_\text{end}}>0$ and the probability space $(\Omega, \mathcal{F}, \mathbb{P})$ for $\Omega$ the space of paths $\left\{ P(t) \right\}_{t\in[0,{t_\text{end}}]}$ such that $P(t)\in L^2(\partial\mathcal{X})$ for all $t\in[0,{t_\text{end}}]$, the Borel $\sigma$-algebra $\mathcal{F}$ on $\Omega$ and a Wiener measure $\mathbb{P}$ induced by a Wiener process on $[0,{t_\text{end}}]$. We associate to such probability space the natural filtration $\mathcal{F}_t$ of $\mathcal{F}$ with respect to a Wiener process and $t\in[0,{t_\text{end}}]$. The cylindrical Wiener process is defined as
\begin{equation} \label{cylindrical_WP}
    W(t)=\sum_{i=1}^\infty b_i \beta_i (t)
\end{equation}
for an orthonormal basis $\left\{b_i\right\}_{i\in\mathbb{N}_{>0}}$ of $L^2(\partial\mathcal{X})$ and $\left\{\beta_i\right\}_{i\in\mathbb{N}_{>0}}$ a family of independent, identically distributed and $\mathcal{F}_t\text{-adapted}$ Brownian motions. There exists hence (\cite[Proposition 4.11]{DP}) a larger space $\tilde{H}$ on which $L^2(\partial \mathcal{X})$ embeds through the Hilbert-Schmidt embedding and for which the series in \eqref{cylindrical_WP} converges in $L^2(\Omega,\mathcal{F},\mathbb{P};\tilde{H})$ for any $t\in[0,{t_\text{end}}]$. \\
The Gaussian white-noise process is labeled as $\dot{W}(x,t)$ and satisfies 
\begin{align*}
    &\mathbb{E}[\dot{W}(x,t)]=0\;,\\
    &\mathbb{E}[\dot{W}(x_1,t_1)\dot{W}(x_2,t_2)]=\delta(x_1-x_2,t_1-t_2)\;,
\end{align*}
for any $x,x_1,x_2\in\partial \mathcal{X}$, any $t,t_1,t_2\geq 0$, for $\mathbb{E}$ the mean value induced by $\mathbb{P}$; $\delta$ denotes the Kronecker delta. The initial condition $u_0\in \mathcal{D}(A(p))$ is assumed to be $\mathcal{F}_0$-measurable. The linear operator
\begin{equation*}
    \gamma(p):\mathcal{D}(\gamma(p))\subseteq \mathcal{H}\mapsto L^2(\partial\mathcal{X})
\end{equation*}
defines the boundary conditions on $\partial\mathcal{X}$ and the bounded linear operator
\begin{equation*}
    B:\mathcal{D}(B)\subseteq L^2(\partial\mathcal{X})\mapsto L^2(\partial\mathcal{X})
\end{equation*}
describes the effect of the noise on the boundary. We define $B^*$ as the adjoint operator of $B$ with respect to the scalar product on the complex Hilbert space $L^2(\partial\mathcal{X})$. We assume the linear operator $A_0$ such that
\begin{equation*}
    A_0(p) v= A(p) v \quad \text{for any } v\in \mathcal{D}(A_0(p))=\mathcal{D}(A(p))\cap \mathcal{D}(\gamma(p))\cap\left\{\gamma(p)\; v(x)=0 \text{ for } x\in\partial
    \mathcal{X}\right\}\;,
\end{equation*}
to have all eigenvalues with negative real part for $p<\lambda\in\mathbb{R}$ and to admit an eigenvalue with zero real part for $p=\lambda$. The operator $A_0(p)$ generates a $C_0-$semigroup on $\mathcal{H}$, $S(t)=e^{A_0(p)\, t}$ for $t\geq0$ and $p<\lambda$. We assume that there exists a constant $c(p)\in\mathbb{R}$ such that $\left(A_0(p)+c(p)\right)^{-1}$, for $p\leq\lambda$, is compact. This entails that the spectrum of $A_0(p)$ is discrete and labeled as $\left\{\lambda_i^{(p)}\right\}_{i\in\mathbb{N}_{>0}}$. For simplicity we assume that $\lambda_1^{(p)}$ is the unique eigenvalue of $A_0(p)$ such that $\operatorname{Re}\left(\lambda_1^{(p)}\right)\to 0^-$ as $p\to\lambda^-$, with $\operatorname{Re}$ that indicates the real part operator. For any $p\leq\lambda$, we refer to $A_0(p)^*$ as the adjoint operator of $A_0(p)$ in respect to the scalar product on the Hilbert space $\mathcal{H}$. We also indicate as $\overline{z}$ the conjugate of the scalar $z\in\mathbb{C}$. The eigenvalues of $A_0(p)^*$ are, hence, $\left\{\overline{\lambda_i^{(p)}}\right\}_{i\in\mathbb{N}_{>0}}$. The operators $A_0(p)$ and $A_0(p)^*$ are assumed to be closed and densely defined in $\mathcal{H}$. We use $m_a$ to indicates the algebraic multiplicity of an eigenvalue. For any $i\in\mathbb{N}_{>0}$, $p\leq\lambda$ and $k\in\left\{1,\dots,m_a\left(\lambda_i^{(p)}\right)\right\}$, the generalized eigenfunctions of $A_0(p)$ and $A_0(p)^*$  corresponding respectively to $\lambda_i^{(p)}$ and $\overline{\lambda_i^{(p)}}$ are labeled as $e_{i,k}^{(p)}$ and ${e_{i,k}^{(p)}}^*$ and satisfy
\begin{align*}
\begin{alignedat}{3}
    & A_0(p) e_{i,k}^{(p)} = \lambda_i^{(p)} e_{i,k}^{(p)},
    && \quad A_0(p)^* {e_{i,k}^{(p)}}^* = \overline{\lambda_k^{(p)}} {e_{i,1}^{(p)}}^*,
    &&\qquad \text{for $k=1$},\\
    & A_0(p) e_{i,k}^{(p)} = \lambda_i^{(p)} e_{i,k}^{(p)} + e_{i,k-1}^{(p)},
    && \quad A_0(p)^* {e_{i,k}^{(p)}}^* = \overline{\lambda_i^{(p)}} {e_{i,k}^{(p)}}^* + {e_{i,k-1}^{(p)}}^*,
    &&\qquad \text{for $k \neq 1$}.
\end{alignedat}
\end{align*}
We assume that such functions are continuous in $\mathcal{H}$ with regard to $p$. The deterministically invariant subspaces generated by the generalized eigenfunctions of $A_0(p)$ and $A_0(p)^*$ associated to the eigenvalue $\lambda_i^{(p)}$ and $\overline{\lambda_i^{(p)}}$ are denoted respectively as $E_i(p)$ and $E_i(p)^*$. Their dimension is labeled as $M_i=m_a\left(\lambda_i^{(p)}\right)$ and is assumed to be independent of $p$. For each $i\in\mathbb{N}_{>0}$, the sets $\left\{e_{i,k}^{(p)}\right\}_{k\in\left\{1,\dots,M_i\right\}}$ and $\left\{{e_{i,M_i-k+1}^{(p)}}^*\right\}_{k\in\left\{1,\dots,M_i\right\}}$ can be chosen to form a biorthogonal system. It follows that, for any $p\leq\lambda$ and $v\in E_i(p)^*$, it holds
\begin{equation} \label{eq:coeff}
    v=\sum_{k=1}^{M_i} \left\langle v, e_{i,M_i-k+1}^{(p)} \right\rangle {e_{i,k}^{(p)}}^*
\end{equation}
and
\begin{equation} \label{eq:coeff_2}
    A_0(p)^* v = \overline{\lambda_i^{(p)}} \left\langle v, e_{i,M_i}^{(p)} \right\rangle {e_{i,1}^{(p)}}^* + \sum_{k=2}^{M_i} \left\langle v, e_{i,M_i-k+1}^{(p)} \right\rangle \left( \overline{\lambda_i^{(p)}} {e_{i,k}^{(p)}}^* + {e_{i,k-1}^{(p)}}^* \right) \quad.
\end{equation}
We also assume that, for fixed $p\leq\lambda$, there exists a continuous $q=q(p)\in\mathbb{R}$ such that for any boundary value problem
\begin{equation*}
    (A(p)-q)w=0 \quad , \quad \gamma(p)\; w=v\;,
\end{equation*}
with $v\in L^2(\partial\mathcal{X})$, there exists a unique solution $w=D(p)v\in\mathcal{D}(A(p))\subseteq \mathcal{H}$. For any $p\leq\lambda$, we label $D(p)^*$ as the adjoint operator of $D(p)$ in respect to the scalar products on the complex Hilbert spaces $\mathcal{H}$ and $L^2(\partial\mathcal{X})$. We assume the operator $D(p)^*$ to be uniformly bounded in $L^2(\mathcal{X};L^2(\partial \mathcal{X}))$ for any $p\leq\lambda$. Therefore, there exists $C^+>0$ such that
\begin{equation} \label{eq:request3}
    \left\lvert\left\langle v, D(p) B B^* D(p)^* w \right\rangle\right\rvert < C^+
\end{equation}
for any $v,w\in \mathcal{H}$. From \cite{balakrishnan2012applied,da1993evolution,ichikawa1985semigroup} it follows that the solution of \eqref{eq:fast_linear} is
\begin{equation} \label{eq:solution}
    u(x,t)=e^{A_0(p) t} u_0(x) + (q-A_0(p))\int_0^t e^{A_0(p) (t-s)} D(p) B \text{d}W_s \in \mathcal{H}
\end{equation}
for $t\in[0,{t_\text{end}}]$ if and only if
\begin{equation*}
    \int_0^t \left\lvert B^* D(p)^* e^{A_0(p)^* s} (A_0(p)^*-q) \right\rvert_{L^2(\mathcal{X};L^2(\partial\mathcal{X}))}^2 \text{d}s<+\infty\;,
\end{equation*}
or equivalently if and only if
\begin{equation} \label{eq:condition}
    \int_0^t \left\lvert (A_0(p)-q) e^{A_0(p) s} D(p) B \right\rvert_{L^2(\partial \mathcal{X};\mathcal{H})}^2 \text{d}s<+\infty\;.
\end{equation}
Under condition \eqref{eq:condition} for any ${t_\text{end}}>0$ and through Ito isometry, the covariance operator of $u$, at time $t\in[0,{t_\text{end}}]$, is therefore
\begin{equation*}
    V_t:=\int_0^t (A_0(p)-q) e^{A_0(p) s} D(p) B B^* D(p)^* e^{A_0(p)^* s} (A_0(p)^*-q) \text{d}s\;.
\end{equation*}
Lastly, we label the time-asymptotic covariance operator as $V_\infty=\underset{t\to\infty}{\lim} V_t$.

\medskip

Section \ref{sec:EWS} below describes the construction of the early-warning signs, as the parameter $p$ approaches a bifurcation threshold, in the form of the rate of divergence of the time-asymptotic autocovariance, of specific observables, and also as the quantitative behaviour of the time-asymptotic autocorrelation along specific directions in $\mathcal{H}$. In order to describe the scaling law of the observables, we employ the big theta notation \cite{knuth1976big}, $\Theta$, i.e., $b_1(p)=\Theta\left(b_2(p)\right)$ if $\underset{p\to \lambda^-}{\lim} \frac{b_1(p)}{b_2(p)}>0$ and $\underset{p\to \lambda^-}{\lim} \frac{b_2(p)}{b_1(p)}>0$ for $b_1,b_2$ locally positive functions.
Section \ref{sec:examples} provides different examples of models on which the early-warning signs predict a bifurcation event.

\section{Construction of the early-warning signs} \label{sec:EWS}

In this section we assume that \eqref{eq:condition} holds for any ${t_\text{end}}>0$ and $p<\lambda$. We construct early-warning signs able to predict the approach of $p$ to $\lambda$ and the corresponding changes ensuing in \eqref{eq:fast_linear}.

\begin{prop} \label{prop:auto_cov}
For any $p<\lambda$, the time-asymptotic autocovariance operator of lag time $\tau\geq0$, defined by the solution \eqref{eq:solution} of \eqref{eq:fast_linear},
\begin{align*}
    & \underset{t\to\infty}{\lim} V_t^\tau \;,
\end{align*}
and such that $V_t^\tau$ satisfies
\begin{align*}
    & \left\langle v, V_t^\tau w\right\rangle= \text{Cov}(\left\langle u(\cdot,t+\tau), v \right\rangle , \left\langle u(\cdot,t), w \right\rangle)
\end{align*}
for any $v,w\in\mathcal{D}\left(A_0(p)^*\right)$ and $t\geq0$, is the operator
\begin{align} \label{eq:tool_for_autocorrelation}
    & V_\infty^\tau= e^{A_0(p) \tau} V_\infty: \mathcal{D}(A_0(p))\to \mathcal{H} \;.
\end{align}
\end{prop}
\begin{proof}
Set $v,w\in\mathcal{D}\left(A_0(p)^*\right)$, then
\begin{align*}
    & \text{Cov}\left(\left\langle u(\cdot,t+\tau), v \right\rangle , \left\langle u(\cdot,t), w \right\rangle\right) \\
    & = \mathbb{E}\left( \left\langle \int_0^{t+\tau} (q-A_0(p)) e^{A_0(p) (t+\tau-s)} D(p) B \text{d}W_s, v \right\rangle 
    \overline{\left\langle \int_0^{t} (q-A_0(p)) e^{A_0(p) (t-s)} D(p) B \text{d}W_s, w \right\rangle} \right) \\
    & = \mathbb{E}\left( \left\langle \int_0^{t} (q-A_0(p)) e^{A_0(p) (t-s)} D(p) B \text{d}W_s, e^{A_0(p)^* \tau} v \right\rangle 
    \overline{\left\langle \int_0^{t} (q-A_0(p)) e^{A_0(p) (t-s)} D(p) B \text{d}W_s, w \right\rangle} \right) \\
    & = \int_0^{t} \left\langle B^* D(p)^* e^{A_0(p)^* (t-s)} (A_0(p)^*-q) e^{A_0(p)^* \tau} v, B^* D(p)^* e^{A_0(p)^* (t-s)} (A_0(p)^*-q) w \right\rangle_{L^2(\partial \mathcal{X})} \text{d}s
    = \left\langle v, e^{A_0(p) \tau} V_{t} w \right\rangle \;.
\end{align*}
Through the limit $t\to\infty$ the proof is concluded.
\end{proof}

Having constructed the time-asymptotic autocovariance operator, we prove that it satisfies a generalized Lyapunov equation.

\begin{lm} \label{lm:lyap}
We set $\tau\geq0$. For any $p<\lambda$, the time-asymptotic autocovariance operator $V_\infty^\tau$ is a solution in $\mathcal{L}_b(\mathcal{H})$, the space of bounded linear operators in $\mathcal{H}$, of the generalized Lyapunov equation
\begin{equation} \label{eq:lyap}
    \left\langle v, A_0(p) V_\infty^\tau w \right\rangle + \left\langle v, V_\infty^\tau A_0(p)^* w \right\rangle = - \left\langle v, e^{A_0(p) \tau} (A_0(p)-q) D(p) B B^* D(p)^* (A_0(p)^*-q) w \right\rangle\;,
\end{equation}
with $v,w\in \mathcal{D}(A_0(p)^*)$.
\end{lm}

\begin{proof}
 The proof follows the method described in \cite[Lemma 2.45]{da2004kolmogorov}. We fix $p<\lambda$. We prove that $V_\infty^\tau$ solves the generalized Lyapunov equation for any $v,w\in\mathcal{D}(A_0(p)^*)$. Through integration by parts and the fact that $A_0(p)^*$ and $A_0(p)^*-q$ commute for any $q\in\mathbb{R}$ we obtain
\begin{align*}
    \left\langle v, A_0(p) V_\infty^\tau w \right\rangle &= \int_0^\infty \left\langle A_0(p)^* v, e^{A_0(p) \tau} (A_0(p)-q) e^{A_0(p) s} D(p) B B^* D(p)^* e^{A_0(p)^* s} (A_0(p)^*-q) w \right\rangle \text{d}s\\
    &= \int_0^\infty \left\langle \frac{\text{d}}{\text{d}s} e^{A_0(p)^* s} v, e^{A_0(p) \tau} (A_0(p)-q) D(p) B B^* D(p)^* e^{A_0(p)^* s} (A_0(p)^*-q) w \right\rangle \text{d}s\\
    &= \left\langle e^{A_0(p)^* s} v, e^{A_0(p) \tau} (A_0(p)-q) D(p) B B^* D(p)^* e^{A_0(p)^* s} (A_0(p)^*-q) w \right\rangle \Bigg\rvert_0^\infty\\
    &- \int_0^\infty \left\langle e^{A_0(p)^* s} v, e^{A_0(p) \tau} (A_0(p)-q) D(p) B B^* D(p)^* e^{A_0(p)^* s} A_0(p)^* (A_0(p)^*-q) w \right\rangle \text{d}s\\
    &= - \left\langle v, e^{A_0(p) \tau} (A_0(p)-q) D(p) B B^* D(p)^* (A_0(p)^*-q) w \right\rangle - \left\langle v, V_\infty^\tau A_0(p)^* w \right\rangle \;.
\end{align*}
We have therefore proven that $V_\infty^\tau$ is a solution of \eqref{eq:lyap}.
\end{proof}

The following theorem introduces an early-warning sign for $p\to\lambda^-$, and therefore for the change of sign of some eigenvalues of $A_0(p)$, as the divergence of $V_\infty^\tau$ along certain directions in $\mathcal{H}$. The proof follows methods described in \cite{bernuzzi2023early,bernuzzi2023bifurcations,KU2,KU}.

\begin{thm} \label{thm:general}
We set $\tau\geq0$.
\begin{itemize}
\item[a)]
For any $i,j\in\mathbb{N}_{>0}$ it holds
\begin{equation} \label{eq:general_EWS}
    \left\langle {e_{i,1}^{(p)}}^*, V_\infty^\tau {e_{j,1}^{(p)}}^* \right\rangle =-\frac{\left(\overline{\lambda_i^{(p)}}-q\right)\left(\lambda_j^{(p)}-q\right)}{
    \left(\overline{\lambda_i^{(p)}}+\lambda_j^{(p)}\right)} e^{\overline{\lambda_i^{(p)}}\tau}
    \left\langle {e_{i,1}^{(p)}}^*,  D(p) B B^* D(p)^* {e_{j,1}^{(p)}}^* \right\rangle 
    = e^{\overline{\lambda_i^{(p)}}\tau} \left\langle {e_{i,1}^{(p)}}^*, V_\infty {e_{j,1}^{(p)}}^* \right\rangle  \;.
\end{equation}

\item[b)]
For $i,j\in\mathbb{N}_{>0}$, we assume that there exist $C^->0$ and $q=q(p)$ that satisfy
\begin{equation} \label{eq:request1}
    \left\lvert \lambda_i -q \right\rvert > C^- \; , \; \left\lvert \lambda_j -q \right\rvert > C^-
\end{equation}
and
\begin{equation} \label{eq:request2}
    \left\lvert\left\langle {e_{i,1}^{(p)}}^*,  D(p) B B^* D(p)^* {e_{j,1}^{(p)}}^* \right\rangle\right\rvert > C^-
\end{equation}
for any $p\leq \lambda$. Then, for $i=j=1$, it holds
\begin{equation*} 
    \underset{p\to\lambda^-}{\lim} \left\lvert\left\langle {e_{1,1}^{(p)}}^*, V_\infty^\tau {e_{1,1}^{(p)}}^* \right\rangle \right\rvert = +\infty \;,
\end{equation*}
and for $(i,j)\in\mathbb{N}_{>0}\times\mathbb{N}_{>0}\setminus(1,1)$, it follows
\begin{equation} \label{eq:thm1_b}
    \left\lvert\left\langle {e_{i,1}^{(p)}}^*, V_\infty^\tau {e_{j,1}^{(p)}}^* \right\rangle \right\rvert
    = \Theta\left(
    1 \right) \;,
\end{equation}
as $p\to\lambda^-$.

\item[c)]

Denoting $e_{i,0}=e_{j,0}\equiv0$ and $k_1\in\left\{1,\dots,M_i\right\},k_2\in\left\{1,\dots,M_j\right\}$, it holds
\begin{align} \label{eq:thm1_1}
    &\left\langle {e_{i,k_1}^{(p)}}^*, V_\infty^\tau {e_{j,k_2}^{(p)}}^* \right\rangle
    = - \frac{ \left\langle {e_{i,k_1}^{(p)}}^*, V_\infty^\tau {e_{j,k_2-1}^{(p)}}^* \right\rangle
    + \left\langle {e_{i,k_1-1}^{(p)}}^*, V_\infty^\tau {e_{j,k_2}^{(p)}}^* \right\rangle}
    {\left(\overline{\lambda_i^{(p)}}+\lambda_j^{(p)}\right)} \nonumber \\
    & - \frac{e^{\overline{\lambda_i^{(p)}} \tau} \underset{k=1}{\overset{k_1}{\sum}} \frac{\tau^{k_1-k}}{(k_1-k)!} \left\langle \left((\overline{\lambda_i^{(p)}}-q)  {e_{i,k}^{(p)}}^*+ {e_{i,k-1}^{(p)}}^*\right), D(p) B B^* D(p)^* \left((\overline{\lambda_j^{(p)}}-q) {e_{j,k_2}^{(p)}}^*+{e_{j,k_2-1}^{(p)}}^*\right) \right\rangle}
    {\left(\overline{\lambda_i^{(p)}}+\lambda_j^{(p)}\right)}  \nonumber 
\end{align}
for any $i,j\in\mathbb{N}_{>0}$.

\item[d)]
We fix $i,j\in\mathbb{N}_{>0}$, $k_1\in\left\{1,\dots,M_i\right\}$ and $k_2\in\left\{1,\dots,M_j\right\}$. We assume that there exists $C^->0$ and $q=q(p)$ that satisfy \eqref{eq:request1} and \eqref{eq:request2} for any $p\leq \lambda$. Then
\begin{equation} \label{eq:thm1_d}
    \left\lvert \left\langle {e_{i,k_1}^{(p)}}^*, V_\infty^\tau {e_{j,k_2}^{(p)}}^*\right\rangle \right\rvert 
    = \Theta\left(
    \left\lvert \left(\overline{\lambda_i^{(p)}}+\lambda_j^{(p)}\right)^{-k_1-k_2+1}\right\rvert\right)
\end{equation}
as $p\to\lambda^-$.
\end{itemize}
\end{thm}

\begin{proof}
\begin{itemize}

\item[a)] We fix the pair of indexes $i,j\in\mathbb{N}_{>0}$. Equation \eqref{eq:lyap} implies that
\begin{align} \label{eq:to_split}
    &\left\langle {e_{i,1}^{(p)}}^*, A_0(p) V_\infty^\tau {e_{j,1}^{(p)}} \right\rangle + \left\langle {e_{i,1}^{(p)}}^*, V_\infty^\tau A_0(p)^* {e_{j,1}^{(p)}}^* \right\rangle\\
    &= - \left\langle {e_{i,1}^{(p)}}^*, e^{A_0(p) \tau} (A_0(p)-q) D(p) B B^* D(p)^* (A_0(p)^*-q)  {e_{j,1}^{(p)}}^* \right\rangle\;. \nonumber
\end{align}
The left-hand side and the right-hand side of \eqref{eq:to_split} satisfy
\begin{align*}
    &\left\langle {e_{i,1}^{(p)}}^*, A_0(p) V_\infty^\tau {e_{j,1}^{(p)}}^* \right\rangle + \left\langle {e_{i,1}^{(p)}}^*, V_\infty^\tau A_0(p)^* {e_{j,1}^{(p)}}^* \right\rangle 
    = \left(\overline{\lambda_i^{(p)}}+\lambda_j^{(p)}\right)\left\langle {e_{i,1}^{(p)}}^*, V_\infty^\tau {e_{j,1}^{(p)}}^* \right\rangle
    \end{align*}
and 
\begin{align*}
    &- \left\langle {e_{i,1}^{(p)}}^*, e^{A_0(p) \tau} (A_0(p)-q) D(p) B B^* D(p)^* (A_0(p)^*-q)  {e_{j,1}^{(p)}}^* \right\rangle \\
    &= - \left\langle (A_0(p)^*-q) e^{A_0(p)^* \tau} {e_{i,1}^{(p)}}^*, D(p) B B^* D(p)^* (A_0(p)^*-q)  {e_{j,1}^{(p)}}^* \right\rangle \\
    &= - \left(\overline{\lambda_i^{(p)}}-q\right) \left(\overline{\lambda_j^{(p)}}-q\right) e^{\overline{\lambda_i^{(p)}}\tau}
    \left\langle {e_{i,1}^{(p)}}^*,  D(p) B B^* D(p)^* {e_{j,1}^{(p)}}^* \right\rangle \;,
\end{align*}
respectively. Hence, equation \eqref{eq:general_EWS} is proven.

\item[b)]

The rate in \eqref{eq:thm1_b} follows from \eqref{eq:request3}, \eqref{eq:lyap}, \eqref{eq:request1} and the fact that $A_0(p)$ has discrete spectrum for any $p\leq\lambda$. In the case $i=j=1$, the divergence is implied by \eqref{eq:request2}.

\item[c)]

The construction of generalized eigenfunctions and \eqref{eq:coeff_2} imply that
\begin{align*}
    &\left\langle {e_{i,k_1}^{(p)}}^*, A_0(p) V_\infty^\tau {e_{j,k_2}^{(p)}}^* \right\rangle + \left\langle {e_{i,k_1}^{(p)}}^*, V_\infty^\tau A_0(p)^* {e_{j,k_2}^{(p)}}^* \right\rangle\\
    &= \left(\overline{\lambda_i^{(p)}}+\lambda_j^{(p)}\right)\left\langle {e_{i,k_1}^{(p)}}^*, V_\infty^\tau {e_{j,k_2}^{(p)}}^* \right\rangle
    + \left\langle {e_{i,k_1}^{(p)}}^*, V_\infty^\tau {e_{j,k_2-1}^{(p)}}^* \right\rangle
    + \left\langle {e_{i,k_1-1}^{(p)}}^*, V_\infty^\tau {e_{j,k_2}^{(p)}}^* \right\rangle \;.
\end{align*}
It also holds
\begin{align*}
    &- \left\langle {e_{i,k_1}^{(p)}}^*, e^{A_0(p) \tau} (A_0(p)-q) D(p) B B^* D(p)^* (A_0(p)^*-q)  {e_{j,k_2}^{(p)}}^* \right\rangle \\
    &= - \left\langle e^{A_0(p)^* \tau} (A_0(p)^*-q) {e_{i,k_1}^{(p)}}^*, D(p) B B^* D(p)^* (A_0(p)^*-q)  {e_{j,k_2}^{(p)}}^* \right\rangle \\
    &= - \left\langle e^{A_0(p)^* \tau} \left(\left(\overline{\lambda_i^{(p)}}-q\right) {e_{i,k_1}^{(p)}}^*+{e_{i,k_1-1}^{(p)}}^*\right), D(p) B B^* D(p)^* \left(\left(\overline{\lambda_j^{(p)}}-q\right) {e_{j,k_2}^{(p)}}^*+{e_{j,k_2-1}^{(p)}}^*\right) \right\rangle \\
    &= - \left\langle e^{\overline{\lambda_i^{(p)}} \tau} \sum_{k=1}^{k_1} \left(\left(\overline{\lambda_i^{(p)}}-q\right) \frac{\tau^{k_1-k}}{(k_1-k)!} {e_{i,k}^{(p)}}^*+\frac{\tau^{k_1-k}}{(k_1-k)!} {e_{i,k-1}^{(p)}}^*\right), D(p) B B^* D(p)^* \left(\left(\overline{\lambda_j^{(p)}}-q\right) {e_{j,k_2}^{(p)}}^*+{e_{j,k_2-1}^{(p)}}^*\right) \right\rangle \\
    &= -  e^{\overline{\lambda_i^{(p)}} \tau} \sum_{k=1}^{k_1} \frac{\tau^{k_1-k}}{(k_1-k)!} \left\langle \left(\left(\overline{\lambda_i^{(p)}}-q\right)  {e_{i,k}^{(p)}}^*+ {e_{i,k-1}^{(p)}}^*\right), D(p) B B^* D(p)^* \left(\left(\overline{\lambda_j^{(p)}}-q\right) {e_{j,k_2}^{(p)}}^*+{e_{j,k_2-1}^{(p)}}^*\right) \right\rangle \;.
\end{align*}
The equality follows from \eqref{eq:lyap}.

\item[d)]
The rate can be proven by induction. For $k_1=k_2=1$, the equality \eqref{eq:general_EWS} and assumptions \eqref{eq:request3}, \eqref{eq:request1} and \eqref{eq:request2} imply that
\begin{align*}
    \left\lvert\left\langle {e_{i,1}^{(p)}}^*, V_\infty^\tau {e_{j,1}^{(p)}}^*\right\rangle \right\rvert 
    = \Theta\left(
    \left\lvert \left(\overline{\lambda_i^{(p)}}+\lambda_j^{(p)}\right)^{-1}\right\rvert\right)\quad.
\end{align*}
The induction step follows directly from \eqref{eq:request3}. From induction, the scaling law is given by the fact that the unique term with leading absolute value in
\begin{equation*}
    \left\langle {e_{i,k_1}^{(p)}}^*, V_\infty^\tau {e_{j,k_2}^{(p)}}^*\right\rangle
\end{equation*}
is
\begin{equation*}
    (-1)^{k_1+k_2} \binom{k_1+k_2-2}{k_1-1} \left\langle {e_{i,1}^{(p)}}^*, V_\infty^\tau {e_{j,1}^{(p)}}^* \right\rangle
    \left(\overline{\lambda_i^{(p)}}+\lambda_j^{(p)}\right)^{-k_1-k_2+2}.
\end{equation*}
\end{itemize}
This finishes the proof.
\end{proof}

Theorem \ref{thm:general} provides an early-warning sign in the form of the qualitative behaviour of the time-asymptotic autocovariance of the solution of system \eqref{eq:fast_linear} along certain directions in $\mathcal{H}$. Of particular importance is the case $i=j=1$ under assumptions \eqref{eq:request1} and \eqref{eq:request2}, which provides divergence of the early-warning sign, whose scaling law can be observed under \eqref{eq:request3}. The restriction on the choice of space on which the autocovariance is studied limits the use of such theory in practical applications. In the following theorem, the use of the time-asymptotic autocovariance operator as an early-warning sign is extended on a bigger subspace of functions in $\mathcal{H}$. To prove such generalization, we introduce the projection operators $\Pi_{1,k}(p)^*$, $\Pi_{1}(p)^*$ and $\Pi_{-1}(p)^*$ in $\mathcal{H}$ as
\begin{align*}
    \Pi_{1,k}(p)^* v= \left\langle v, e_{1,M_1-k+1}^{(p)} \right\rangle {e_{1,k}^{(p)}}^*
    \quad \text{,} \quad
    \Pi_{1}(p)^* v=
    \sum_{k=1}^{M_1} \Pi_{1,k}(p)^* v
    \qquad \text{and} \qquad
    \Pi_{-1}(p)^* v= v-\Pi_{1}(p)^* v
    \quad ,
\end{align*}
for any $v\in \mathcal{H}$ and $k\in\left\{1,\dots,M_1\right\}$.

\begin{thm} \label{thm:second}
We set $\tau\geq0$ and $M\in\mathbb{N}_{>0}$. We define $O_M(p)^*:=\overset{M}{\underset{i=1}{\bigoplus}} E_i(p)^*$.

\begin{itemize}
\item[a)]
For any $f_1,f_2\in O_M(p)^*$ it holds

\begin{equation} \label{eq:second_EWS}
    \left\langle f_1, V_\infty^\tau f_2 \right\rangle =
    \underset{i,j\in\left\{1,\dots,M\right\}}{\mathlarger{\mathlarger\sum}}
    \underset{k_1\in\left\{1,\dots,M_i\right\}}{\mathlarger{\mathlarger\sum}}
    \underset{k_2\in\left\{1,\dots,M_j\right\}}{\mathlarger{\mathlarger\sum}}
    \left\langle f_1, e_{i,M_i-k_1+1}^{(p)} \right\rangle
    \left\langle {e_{i,k_1}^{(p)}}^*, V_\infty^\tau {e_{j,k_2}^{(p)}}^* \right\rangle
    \overline{\left\langle f_2, e_{j,M_j-k_2+1}^{(p)} \right\rangle}\;.
\end{equation}

\item[b)]
We assume the existence of $C^->0$ and $q=q(p)$ that satisfy \eqref{eq:request1} for all $i,j\in\{1,\dots,M\}$ and \eqref{eq:request2} for $i,j=1$. We set the sequences $\left\{f_1^{(p)}\right\},\left\{f_2^{(p)}\right\}$ continuous in $\mathcal{H}$, for $p\leq\lambda$, such that $\Pi_{1}(\lambda)^* f_1^{(\lambda)}\not\equiv0\not\equiv\Pi_{1}(\lambda)^* f_2^{(\lambda)}$ and $f_1^{(p)},f_2^{(p)}\in O_M(p)^*$ for any $p\leq\lambda$. Then it holds
\begin{equation*}
    \underset{p\to\lambda^-}{\lim}\left\lvert\left\langle f_1^{(p)}, V_\infty^\tau f_2^{(p)}\right\rangle \right\rvert
    = +\infty \;.
\end{equation*}
\item[c)]
Under the assumption of the previous item and for fixed $k_1,k_2\in\left\{1,\dots,M_1\right\}$ such that 
\begin{equation} \label{eq:request_4}
    k_m=
    \underset{i\in\left\{1,\dots,M_1\right\}}{\text{argmax}}\left\{\Pi_{1,j}(\lambda)^* f_m^{(\lambda)}\not\equiv0 \right\}\;,
\end{equation}
    for $m\in\{1,2\}$, it holds
\begin{equation} \label{eq:main_order_cor}
    \left\lvert\left\langle f_1^{(p)}, V_\infty^\tau f_2^{(p)}\right\rangle \right\rvert
    = \Theta\left(\operatorname{Re}\left( -\lambda_1^{(p)}\right)^{-k_1-k_2+1} \right) \;.
\end{equation}
\end{itemize}

\end{thm}

\begin{proof}
\begin{itemize}
    \item[a)] 
    The construction of $O_M(p)^*$ implies that $f_1$ and $f_2$ are linear combinations of the generalized eigenfunctions of $A_0(p)^*$ associated to $\overline{\lambda_i^{(p)}}$ for $i\in\left\{1,\dots,M\right\}$. The coefficients in such sum follow from \eqref{eq:coeff}. It therefore holds
    \begin{align*}
        \left\langle f_1, V_\infty^\tau f_2 \right\rangle 
        &=
        \underset{i,j\in\left\{1,\dots,M\right\}}{\mathlarger{\mathlarger\sum}}
        \left\langle \underset{k_1\in\left\{1,\dots,M_i\right\}}{\mathlarger{\mathlarger\sum}} \left\langle f_1, e_{i,M_i-k_1+1}^{(p)} \right\rangle {e_{i,k_1}^{(p)}}^*, V_\infty^\tau
        \underset{k_2\in\left\{1,\dots,M_j\right\}}{\mathlarger{\mathlarger\sum}} \left\langle f_2,  e_{j,M_j-k_2+1}^{(p)} \right\rangle {e_{j,k_2}^{(p)}}^* \right\rangle \\
         &=
        \underset{i,j\in\left\{1,\dots,M\right\}}{\mathlarger{\mathlarger\sum}}
        \underset{k_1\in\left\{1,\dots,M_i\right\}}{\mathlarger{\mathlarger\sum}}
        \underset{k_2\in\left\{1,\dots,M_j\right\}}{\mathlarger{\mathlarger\sum}}
        \left\langle f_1, e_{i,M_i-k_1+1}^{(p)} \right\rangle
        \left\langle {e_{i,k_1}^{(p)}}^*, V_\infty^\tau {e_{j,k_2}^{(p)}}^* \right\rangle
        \overline{\left\langle f_2, e_{j,M_j-k_2+1}^{(p)} \right\rangle}\;.
    \end{align*}
    
    \item[b)] 
    Assumption \eqref{eq:request3} implies that, for any $p\leq\lambda$, it holds
    \begin{equation} \label{eq:condition2}
        \int_0^\infty \left\lvert\left\lvert B^* D(p)^* e^{A_0(p)^* s} \left(A_0(p)^*-q\right)  \Pi_{-1}(p)^* f \right\rvert\right\rvert_{L^2(\partial \mathcal{X})}^r \text{d}s\;<+\infty\;,
    \end{equation}
    for any $f\in O_M(p)^*$ and $r>0$.\\
    Since $f_1^{(p)},f_2^{(p)}\in \mathcal{H}$, it holds for any $p\leq\lambda$
    \begin{align*}
        \left\langle f_1^{(p)}, V_\infty^\tau f_2^{(p)} \right\rangle 
        &= \left\langle (\Pi_{-1}(p)^*+\Pi_1(p)^*) f_1^{(p)}, V_\infty^\tau (\Pi_{-1}(p)^*+\Pi_1(p)^*) f_2^{(p)} \right\rangle \\
        &= \left\langle \Pi_{-1}(p)^* f_1^{(p)}, V_\infty^\tau \Pi_{-1}(p)^* f_2^{(p)} \right\rangle 
        + \left\langle \Pi_{-1}(p)^* f_1^{(p)}, V_\infty^\tau \Pi_1(p)^* f_2^{(p)} \right\rangle \\
        &+ \left\langle \Pi_1(p)^* f_1^{(p)}, V_\infty^\tau \Pi_{-1}(p)^* f_2^{(p)} \right\rangle 
        + \left\langle \Pi_1(p)^* f_1^{(p)}, V_\infty^\tau \Pi_1(p)^* f_2^{(p)} \right\rangle \;.
    \end{align*}
    From \eqref{eq:condition2} it is implied that
    \begin{align} \label{eq:easy_if_D_bounded}
         &\left\lvert \left\langle \Pi_{-1}(p)^* f_1^{(p)}, V_\infty^\tau \Pi_{-1}(p)^* f_2^{(p)} \right\rangle \right\rvert \\
         & =\left\lvert \int_0^\infty \left\langle B^* D(p)^* e^{A_0(p)^* s} (A_0(p)^*-q) e^{A_0(p)^* \tau} \Pi_{-1}(p)^*  f_1^{(p)}, B^* D(p)^* e^{A_0(p)^* s} (A_0(p)^*-q)  \Pi_{-1}(p)^* f_2^{(p)} \right\rangle_{L^2(\partial \mathcal{X})} \text{d}s  \right\rvert <+\infty \nonumber
    \end{align}
    for any $p\leq\lambda$. In the following we use the fact that, for any $i\in\mathbb{N}_{>0}$, $E_i(p)^*$ is an invariant subspace of $\mathcal{D}(A_0(p)^*)$ under the action of $A_0(p)^*$ and that $A_0(p)^*\lvert_{E_i}$ is a bounded non-positive operator in $\mathcal{D}(A_0(p)^*)$ by construction. Also, condition \eqref{eq:condition} and \eqref{eq:condition2} imply that, for $p$ approaching $\lambda$ from below,
    \begin{align} \label{eq:easy_2}
         & \left\lvert \left\langle \Pi_{-1}(p)^* f_1^{(p)}, V_\infty^\tau \Pi_1(p)^* f_2^{(p)} \right\rangle \right\rvert \\
         & =\left\lvert \int_0^\infty \left\langle B^* D(p)^* e^{A_0(p)^* s} (A_0(p)^*-q) e^{A_0(p)^* \tau} \Pi_{-1}(p)^* f_1^{(p)}, B^* D(p)^* e^{A_0(p)^* s} (A_0(p)^*-q)  \Pi_1(p)^* f_2^{(p)} \right\rangle_{L^2(\partial \mathcal{X})} \text{d}s \right\rvert \nonumber \\
         & =\left\lvert
         \int_0^\infty \left\langle B^* D(p)^* e^{A_0(p)^* s} (A_0(p)^*-q) \Pi_{-1}(p)^* e^{A_0(p)^* \tau} f_1^{(p)}, B^* D(p)^* e^{A_0(p)^*\lvert_{E_1} s} (A_0(p)^*\lvert_{E_1}-q) \Pi_{1}(p)^* f_2^{(p)} \right\rangle_{L^2(\partial \mathcal{X})} \text{d}s \right\rvert \nonumber \\
         & \leq C
         \int_0^\infty \Big\lvert\Big\lvert B^* D(p)^* e^{A_0(p)^* s} (A_0(p)^*-q) \Pi_{-1}(p)^* e^{A_0(p)^* \tau} f_1^{(p)} \Big\rvert\Big\rvert_{L^2(\partial \mathcal{X})} \; \text{d}s \;
         \sum_{k=1}^{M_1}\Big\lvert\Big\lvert B^* D(p)^* {e_{1,k}^{(p)}}^* \Big\rvert\Big\rvert_{L^2(\partial \mathcal{X})} = \Theta \left(1\right) \nonumber \;.
    \end{align} 
    Equivalently,
    \begin{align} \label{eq:easy_3}
         \left\lvert \left\langle \Pi_1(p)^* f_1^{(p)}, V_\infty^\tau \Pi_{-1}(p)^* f_2^{(p)} \right\rangle \right\rvert
         & = \Theta \left(1\right)
    \end{align}    
    for $p \to \lambda^-$. Lastly, from \eqref{eq:second_EWS} we obtain
    \begin{align} \label{eq:tool_abs_sum}
         \left\lvert \left\langle \Pi_1(p)^* f_1^{(p)}, V_\infty^\tau \Pi_1(p)^* f_2^{(p)} \right\rangle \right\rvert
         & = \left\lvert \underset{k_1,k_2\in\left\{1,\dots,M_1\right\}}{\mathlarger{\mathlarger\sum}} \left\langle f_1^{(p)}, e_{1,k_1}^{(p)} \right\rangle  \;
         \left\langle {e_{1,k_1}^{(p)}}^*, V_\infty^\tau {e_{1,k_2}^{(p)}}^* \right\rangle \;
         \left\langle f_2^{(p)}, e_{1,k_2}^{(p)} \right\rangle \right\rvert \;.
    \end{align}
    From \eqref{eq:thm1_d} and the continuity of the generalized eigenfunctions on variable $p$ follows the divergence.
    
    \item[c)]
    The previous item and \eqref{eq:tool_abs_sum} imply 
    \begin{align*}
         \left\lvert \left\langle \Pi_1(p)^* f_1^{(p)}, V_\infty^\tau \Pi_1(p)^* f_2^{(p)} \right\rangle \right\rvert
         & = \Theta\left(\operatorname{Re}\left(- \lambda_1^{(p)}\right)^{-k_1-k_2+1} \right)
    \end{align*}
    and the conclusion of the proof.
    \end{itemize} 
\end{proof}

Theorem \ref{thm:second} presents an early-warning sign whose numerical approximation (see Section \ref{sec:examples}) does not rely on the precise computation of the elements in $E_1(p)$ and $E_1(p)^*$ for $p\leq\lambda$. Simulations and implementations of such results are more practical than the early-warning signs obtained in Theorem \ref{thm:general}.

\begin{remark}
The assumption of boundary noise can affect the rate of divergence of the autocovariance along certain directions in $\mathcal{D}(A_0(p))^*$. Furthermore, the omission of assumption \eqref{eq:request2} can remove the early-warning sign. For example, in the case 
\begin{equation*}
    \left\langle {e_{1,k_1}^{(p)}}^*,  D(p) B B^* D(p)^* {e_{1,k_2}^{(p)}}^* \right\rangle = 0
\end{equation*}
for any $p\leq \lambda$ and $k_1,k_2\in\{1,\dots,M_1\}$, it is implied that the observable $\left\lvert\left\langle f_1^{(p)}, V_\infty^\tau f_2^{(p)}\right\rangle \right\rvert$, for $\left\{f_1^{(p)}\right\}$ and $\left\{f_2^{(p)}\right\}$ that satisfy the assumptions in Theorem \ref{thm:second}, does not display divergence as $p\to\lambda^-$. This is caused by the interplay of the linear operator $A(p)$ that defines the drift component, the boundary operator $\gamma(p)$ and the operator $B$ associated to the diffusion. In detail, the first two define $A_0(p)$, $A_0(p)^*$ and their eigenfunctions, along with the operators $D(p)$ and $D(p)^*$. The observable can display convergence also in the case $\left\{f_1^{(p)}\right\}$ or $\left\{f_2^{(p)}\right\}$ are orthogonal to $E_1(p)$, as implied by \eqref{eq:second_EWS}. A similar property is employed in Section \ref{ssec:Bous}. The role of the choice of directions in $\mathcal{H}$, along which the time-asymptotic autocovariance is studied, is discussed further in the following corollary. 
\end{remark}

\begin{cor} \label{cor:autocorrelation}
We set $\tau\geq0$. 
\begin{itemize}
\item[a)] We assume the existence of $C^->0$ and $q=q(p)$ that satisfy \eqref{eq:request1} for all $i,j\in\mathbb{N}_{>0}$ and \eqref{eq:request2} for $i,j=1$. We assume also that the generalized eigenfunctions of $A_0(p)^*$ are complete in $\mathcal{D}(A_0(p)^*)$ for any $p\leq\lambda$. Lastly, we set the sequences $\left\{f_1^{(p)}\right\},\left\{f_2^{(p)}\right\}$ continuous in $\mathcal{H}$ for $p\leq\lambda$. Then for any $\delta>0$ there exist two sequences $\left\{g_1^{(p)}\right\},\left\{g_2^{(p)}\right\}$ continuous in $\mathcal{H}$ such that $g_1^{(p)},g_2^{(p)}\in\mathcal{D}(A_0(p)^*)$,
\begin{equation} \label{eq:shadow_cor}
    \left\lvert\left\lvert f_1^{(p)}-g_1^{(p)} \right\rvert\right\rvert<\delta \quad,\quad
    \left\lvert\left\lvert f_2^{(p)}-g_2^{(p)} \right\rvert\right\rvert<\delta\;,
\end{equation}
for any $p\leq\lambda$, and
\begin{equation*}
    \left\lvert\left\langle g_1^{(p)}, V_\infty^\tau g_2^{(p)}\right\rangle \right\rvert
    = \Theta\left(\operatorname{Re}\left(- \lambda_1^{(p)}\right)^{-2 M_1 +1} \right)
\end{equation*}
for $p\to\lambda^-$.
\item[b)]
We set $p<\lambda$ and assume that the generalized eigenfunctions of $A_0(p)^*$ are complete in $\mathcal{D}(A_0(p)^*)$. We define the autocorrelation nonlinear operator at time $t>0$ and lag $\tau\geq0$ as
\begin{equation*}
    \hat{V}_t^\tau(v, w)
    := \frac{\left\langle v, V_t^\tau w \right\rangle}{\left\langle v, V_t w \right\rangle}
\end{equation*}
for any $v,w\in\mathcal{D}(A_0(p)^*)$ such that $\left\langle v, V_t w \right\rangle\neq0$. The time-asymptotic autocorrelation nonlinear operator of lag $\tau\geq0$, labeled $\hat{V}_\infty^\tau$ and defined as
\begin{equation*}
    \hat{V}_\infty^\tau(v, w)
    = \underset{t\to\infty}{\lim}\hat{V}_t^\tau(v, w)
\end{equation*}
for any $v,w\in\mathcal{D}(A_0(p)^*)$ such that $\left\langle v, V_\infty w \right\rangle\neq0$, satisfies
\begin{equation} \label{eq:autocorr_final}
    \hat{V}_\infty^\tau \left({e_{i,1}^{(p)}}^*, f\right) = e^{\overline{\lambda_i^{(p)}}\tau},
\end{equation}
for any $i\in\mathbb{N}_{>0}$ and $f$ in a dense subset $\mathcal{H}'$ of $\mathcal{H}$ such that $\left\langle {e_{i,1}^{(p)}}^*, V_\infty f \right\rangle\neq0$.
\end{itemize}
\end{cor}

\begin{proof}
\begin{itemize}
\item[a)] For $p\leq \lambda$, the fact that the generalized eigenfunctions of $A_0(p)^*$ are complete in its domain implies that the set of finite linear combinations of such functions, labeled $\mathcal{H}_0(p)$, is dense in $\mathcal{D}(A_0(p)^*)$. We also denote the set of functions $f\in\mathcal{H}$ such that $\Pi_{1,M_1}(p)^* f\not\equiv0$ as $\mathcal{H}_1(p)$. By construction, $\mathcal{H}_1(p)$ is dense in $\mathcal{H}$ and open. Therefore $\mathcal{H}'(p):=\mathcal{H}_0(p)\cap\mathcal{H}_1(p)$ is dense in $\mathcal{D}(A_0(p)^*)$, which in return is dense in $\mathcal{H}$.\\
From the continuity of $\left\{f_1^{(p)}\right\},\left\{f_2^{(p)}\right\}$ and the generalized eigenfunctions of $A_0(p)$ and $A_0(p)^*$, we can construct $\left\{g_1^{(p)}\right\},\left\{g_2^{(p)}\right\}$ continuous in $\mathcal{H}$ such that $g_1^{(p)},g_2^{(p)}\in\mathcal{H}'(p)$ for any $p\leq\lambda$. In particular, there exists $M$, dependent on $\delta$, such that $g_1^{(p)},g_2^{(p)}\in O_M(p)^*$ for any $p\leq\lambda$.
The conclusion is obtained from Theorem \ref{thm:second}.

\item[b)] We construct $\mathcal{H}'=\mathcal{H}_0(p)$, as in the previous part of the proof, for the chosen $p<\lambda$. The conclusion holds directly from \eqref{eq:tool_for_autocorrelation} and \eqref{eq:second_EWS}, fixing $f_1={e_{i,1}^{(p)}}^*$ and $f_2=f\in\mathcal{H}'$. 
\end{itemize}
\end{proof}

Corollary \ref{cor:autocorrelation} for $\tau>0$ presents an early-warning sign of the approach to the bifurcation threshold $\lambda$ as the time-asymptotic autocorrelation, along ${{e_{1,1}}^{(p)}}^*$ and another general function, reaches a value $1$ with an exponential rate. The behaviour of the autocorrelation is not trivial to observe numerically since such a sign is given by a quantitative change, in contrast to the time-asymptotic autocovariance, which is expected to diverge and hence change qualitatively. Such divergence is observed for a wide variety of directions in $\mathcal{H}$. This is implied by the assumption that the generalized eigenfunctions of $A_0(\lambda)^*$ are dense in $\mathcal{H}$. Such property has been proven for different non-self-adjoint operators with compact resolvent \cite{agmon1962eigenfunctions, inbook,zhang2001completeness}.

\begin{remark} \label{rmk:autocorrelation}
The time-asymptotic autocorrelation is a well-known early-warning sign for finite dimensional systems \cite{ditlevsen2010tipping}. Setting $\tau>0$ and assuming $M_1=1$, the autocorrelation $\hat{V}_\infty^\tau\left(f_1^{(p)},f_2^{(p)}\right)$, for functions $f_1^{(p)},f_2^{(p)}\in O_M(p)^*$ and $M\in\mathbb{N}_{>0}$, is yet not a straightforward object to determine analytically. From the sum in \eqref{eq:second_EWS}, we know that, for ${\Pi_1(p)}^* f_1^{(p)} \not\equiv 0 \not\equiv {\Pi_1(p)}^* f_2^{(p)}$, the leading term in the outer sum as $p\to\lambda^-$ is the component associated to $i=j=1$. This implies that for values of $p$ such that the mentioned element assumes a higher order of magnitude than the rest, and consequently the rates in Theorem \ref{thm:second} $(c)$ are observed, the time-asymptotic autocorrelation behaves similarly to $e^{\overline{\lambda_1^{(p)}}\tau}$ and assumes absolute value equal to $1$ on the threshold.
\end{remark}

\begin{remark} \label{rem:generalization}
In Theorem \ref{thm:general}, Theorem \ref{thm:second} and Corollary \ref{cor:autocorrelation}, we assume the operator $A_0(p)$ to have discrete spectrum and $\lambda_1^{(p)}$ to be the only eigenvalue whose real part changes sign as $p\to\lambda^-$. Such properties can be generalized by considering that the set of eigenvalues $\Lambda=\left\{\lambda_i^{(p)}\right\}_{i}$ such that
\begin{equation*}
    \operatorname{Re}\left(\lambda_i^{(p)}\right)\to 0 \qquad \text{as} \qquad p\to\lambda^-
\end{equation*}
is finite and isolated. This implies an asymptotic gap between the elements of such a set and the rest of the spectrum. Under such assumption, the divergence in Theorem \ref{thm:general} (b) is observed along the corresponding simple eigenfunctions of any element in $\Lambda$ and the divergence in Theorem \ref{thm:second} (b) is obtained along $\left\{f_1^{(p)}\right\}$ and $\left\{f_2^{(p)}\right\}$ whose projection on the generalized eigenspace of any element in $\Lambda$ is not null for $p=\lambda$. The rates Theorem \ref{thm:second} (c) and Corollary \ref{cor:autocorrelation} (a) depend on the rate of convergence of the real part of the elements in $\Lambda$. Certain elliptic (differential) operators are known to display discrete spectrum and completeness of the generalized eigenfunctions \cite{agmon1962eigenfunctions,inbook}. For the property of boundedness of $D(p)$ for such operators, we refer to \cite{lasiecka1980unified}. We emphasize that Theorem \ref{thm:second} (b) and (c) do not require the spectrum of $A_0(p)^*$ to be discrete, but rely on the fact that the set $\left\{\overline{\lambda_i^{(p)}}\right\}_{i}$ is composed by isolated points from the rest of the spectrum.
\end{remark}

\section{Examples and applications} \label{sec:examples}

The current section provides several examples of models on which the constructed early-warning signs are applicable.

\subsection{The heat equation with noise boundary conditions}

\begin{example} \label{ex:neumann}
We set $\mathcal{X}=[0,L]^N$, for $N\in\left\{1,2,3\right\}$, and $\gamma(p)=\gamma_\text{N}$ in \eqref{eq:fast_linear} that defines homogeneous Neumann boundary conditions on $\mathcal{X}$. We also consider $A(p)=\Delta+p$ and $A_0(p)=\Delta_\text{N}+p$ for $\Delta$ the Laplace operator on $\mathcal{X}$ and $\Delta_\text{N}$ the Laplace operator with homogeneous Neumann boundary conditions, i.e.,
\begin{equation*}
    \Delta_\text{N} f = \Delta f \text{, for any }
    f\in\mathcal{D}(\Delta_\text{N}):=\left\{v\in \mathcal{H}^2(\mathcal{X}):\; \gamma_\text{N}\; v(x)=0 \;,\;x\in\partial\mathcal{X}\right\}\;.
\end{equation*}
It is well known that the non-positive self-adjoint operator $\Delta_\text{N}$ has discrete spectrum $\left\{-\hat{\lambda}_i\right\}_{i\in\mathbb{N}_{>0}}$ and that its eigenfunctions form a basis in $\mathcal{H}$.
Due to the construction of $A_0(p)$, the domain $\mathcal{D}(A_0(p))$ does not depend on $p$. We note that the smallest value in $\left\{\hat{\lambda}_i\right\}_{i\in\mathbb{N}_{>0}}$ is $\hat{\lambda}_1=0$ and that it has multiplicity $1$, hence in this case $\lambda=0$ and $M_1=1$. Furthermore, for any $p<0$ we set $q=q(p)=p+c$ for a fixed $c>0$. Such choice implies that $A(p)-q=\Delta_\text{N}-c$ and that the operator $D$ is independent of $p$. From the shape of $\mathcal{X}$ we know that $D$ is also bounded \cite{da1996ergodicity}.\\
It is also known, from \cite[Theorem 13.3.6]{da1996ergodicity}, that, for certain operators $B$, the conditions \eqref{eq:condition} and \eqref{eq:condition2} hold. It follows that the solution of \eqref{eq:fast_linear} assumes values in $\mathcal{H}$. We can then write the covariance operator at time $t>0$ as
 \begin{equation*}
    V_t:=\int_0^t (\Delta_\text{N}-c) e^{(\Delta_\text{N}+p) s} D(p) B B^* D(p)^* e^{(\Delta_\text{N}+p) s} (\Delta_\text{N}-c) \text{d}s\;.
\end{equation*}
We can also state that Theorem \ref{thm:general}, Theorem \ref{thm:second} and Corollary \ref{cor:autocorrelation} hold. In particular for any $f_1,f_2\in\mathcal{D}\left(\Delta_\text{N}\right)$ such that $\Pi_{1}(\lambda)^* f_1 \not\equiv 0 \not\equiv \Pi_{1}(\lambda)^* f_2$ and $\tau\geq0$,
\begin{equation*}
    \left\lvert\left\langle f_1, V_\infty^\tau f_2\right\rangle \right\rvert
    = \Theta\left(-\frac{1}{ p} \right)
\end{equation*}
as $p\to 0^{-}$.
\end{example}

\begin{example} \label{ex: dirichlet}

We fix $\mathcal{X}=[0,L]$ and $\gamma(p)=\gamma_\text{D}$ in \eqref{eq:fast_linear} such that $\gamma_\text{D}$ requires the solutions of the system to satisfy homogeneous Dirichlet boundary conditions on the interval. Similarly to the previous example, we set $A(p)=\Delta+p$ and $A_0(p)=\Delta_\text{D}+p$, for 
\begin{equation*}
    \Delta_\text{D} f = \Delta f \text{, for any }
    f\in\mathcal{D}(\Delta_\text{D}):=\left\{v\in \mathcal{H}^2(\mathcal{X}):\; \gamma_\text{D}\; v(x)=0 \;,\;x\in\partial\mathcal{X}\right\}\;.
\end{equation*}
The negative self-adjoint operator $\Delta_\text{D}$ has discrete spectrum $\left\{-\hat{\lambda}_i\right\}_{i\in\mathbb{N}_{>0}}$ coupled to the eigenbasis $\left\{\hat{e}_i\right\}_{i\in\mathbb{N}_{>0}}$ in $\mathcal{H}$. The domain $\mathcal{D}(A_0(p))$ is independent of $p$. We label the smallest value in $\left\{\hat{\lambda}_i\right\}_{i\in\mathbb{N}_{>0}}$ as $\hat{\lambda}_1$, which has multiplicity $1$. Hence, in this case, the threshold is $\lambda=\hat{\lambda}_1$ and $M_1=1$. For any $p<\hat{\lambda}_1$ we set $q=q(p)=p+c$ for a fixed $c>0$, which implies that $A(p)-q=\Delta_\text{D}-c$ and that the operator $D$ is independent of $p$. We know from \cite{da1993evolution} that $D$ is bounded, yet \eqref{eq:fast_linear} does not have solutions in $\mathcal{H}$. The system admits nevertheless solutions in $\mathcal{H}^\alpha(\mathcal{X})$, the Sobolev space of degree $\alpha$, for $\alpha<-\frac{1}{2}$. The fact that $\alpha$ is required to be negative implies that \eqref{eq:condition} does not hold. Following the same step as the proof of Lemma \ref{lm:lyap}, a generalized Lyapunov equation of the form
\begin{align*} 
    & \left\langle (-A_0(p)) U(p) v, V_\infty^\tau U(p) w \right\rangle
    + \left\langle U(p) v, V_\infty^\tau (-A_0(p)) U(p) w \right\rangle \\
    &= - \left\langle U(p) v, e^{A_0(p) \tau} (A_0(p)-q) D(p) B B^* D(p)^* (A_0(p)-q) U(p) w \right\rangle\;, \nonumber
\end{align*}
can be proven for weight $U(p)=(-A_0(p))^{\frac{\alpha}{2}}$, functions $U(p) v, U(p) w \in \mathcal{D}(A_0(p))$ and any $\tau\geq0$. We can then obtain
\begin{equation*}
    \left\langle U(p) \hat{e}_i, V_\infty^\tau U(p) \hat{e}_j \right\rangle =-\frac{e^{-\hat\lambda_i \tau} \left(-\hat{\lambda}_i-c\right)\left(-\hat{\lambda}_i-c\right)
    \left((-\hat{\lambda}_i +p) (-\hat{\lambda}_j +p)\right)^{\frac{\alpha}{2}}
    }{
    \left(-\hat{\lambda}_i-\hat{\lambda}_i+2p\right)}
    \left\langle \hat{e}_i,  D(p) B B^* D(p)^* \hat{e}_j \right\rangle \;.
\end{equation*}
Following the steps in Theorem \ref{thm:general} and Theorem \ref{thm:second}, it follows
\begin{equation} \label{eq:dirichlet_1}
    \left\lvert\left\langle U(p) \hat{e}_i, V_\infty^\tau U(p) \hat{e}_j\right\rangle \right\rvert 
    = \Theta\left(\left\lvert 
    \frac{\left((-\hat{\lambda}_i +p) (-\hat{\lambda}_j +p)\right)^{\frac{\alpha}{2}}
    }{
    \left(-\hat{\lambda}_i-\hat{\lambda}_i+2p\right)}
    \right\rvert\right)
\end{equation}
and, for any $U(p) f_1, U(p) f_2\in\mathcal{D}(A_0(p))$ such that $\Pi_1(\lambda)^* f_1 \not\equiv 0 \not\equiv \Pi_1(\lambda)^* f_2$, $f_1$ and $f_2$ are finite linear combinations of eigenfunctions $\left\{\hat{e}_i\right\}_{i\in\mathbb{N}_{>0}}$ and $\tau\geq0$,
\begin{equation} \label{eq:dirichlet_2}
    \left\lvert\left\langle U(p) f_1, V_\infty^\tau U(p) f_2\right\rangle \right\rvert
    = \Theta\left(\left\lvert 
    \left( \hat{\lambda}_1 -p \right)^{-1+\alpha}
    \right\rvert\right)
\end{equation}
as $p\to\hat{\lambda}_1^-$.
\end{example}

\begin{remark}
Weighting the norm in $\mathcal{H}$ through the use of the operator $U(p)=(-A_0(p))^{\frac{\alpha}{2}}$ leads to early-warning signs in \eqref{eq:dirichlet_1} and \eqref{eq:dirichlet_2} with a higher rate of divergence than the ones presented in Theorem \ref{thm:general} and Theorem \ref{thm:second}. In order to avoid such behaviour induced by the convergence of the highest eigenvalue of $A_0(p)$ to the imaginary axis, the operator $(-A_0(p)+q)^{\frac{\alpha}{2}}=(-\Delta_\text{D}+c)^{\frac{\alpha}{2}}=U$ can be chosen as a weight. From \cite{da1993evolution} and assuming
\begin{align*}
    & g_1,g_2\in \mathcal{H}\\
    &f_1=(-\Delta_\text{D}+c)^{-\frac{1}{4}} g_1\in\mathcal{D}(A_0(p))\;,\\
    &f_2=(-\Delta_\text{D}+c)^{-\frac{1}{4}}g_2\in\mathcal{D}(A_0(p))\;,
\end{align*} 
it holds $\left\langle f_1, V_\infty^\tau f_2 \right\rangle<\infty$ for any $p<\lambda$. A numerical analysis of the early-warning signs is possible as the functions $f_1$ and $f_2$ can be assumed to approximate elements in $O_M(p)^*$, for a sufficiently large $M\in\mathbb{N}_{>0}$ and any $p<\lambda$, due to the completeness of the eigenfunctions of $A_0(p)$ in $\mathcal{H}$.
\end{remark}

\subsection{From theory to practical applications}
The early-warning signs presented in Corollary \ref{cor:autocorrelation} generalize the tools employed in finite-dimensional models as they consider the effect of spatial components and noise on the boundary of the domain. As an example, such early-warning signs are employed in climate science. Specific models study the collapse of the Atlantic Meridional Overturning Circulation \cite{boers2021observation,ditlevsen2023warning,van2024physics} and the melting of the Western Greenland Ice Sheet \cite{boers2021critical}. Such tipping points are often associated to fold bifurcations.\medskip

The application of the time-asymptotic autocovariance and autocorrelation as early-warning signs on real-life models relies on Birkhoff's Ergodic Theorem, which states that, for fixed $p<\lambda$ and ergodic solutions $u$ of \eqref{eq:fast_linear}, the time-asymptotic autocovariance equals the asymptotic temporal autocovariance, i.e.,
\begin{equation*}
    \left\langle f_1, V_\infty^\tau f_2 \right\rangle 
    = \int_0^\infty \left\langle u(\cdot,s+\tau),f_1\right\rangle 
    \left\langle u(\cdot,s),f_2\right\rangle
    \text{d}s\;,
\end{equation*}
and the time-asymptotic autocorrelation corresponds to the asymptotic temporal autocorrelation, i.e.,
\begin{equation*}
    \hat{V}_\infty^\tau\left( f_1, f_2 \right) 
    = \frac{\int_0^\infty \left\langle u(\cdot,s+\tau),f_1\right\rangle 
    \left\langle u(\cdot,s),f_2\right\rangle
    \text{d}s}{\int_0^\infty \left\langle u(\cdot,s),f_1\right\rangle 
    \left\langle u(\cdot,s),f_2\right\rangle
    \text{d}s}
    \;,
\end{equation*}
for any $f_1,f_2\in\mathcal{D}(A_0(p)^*)$ such that $\left\langle f_1, V_\infty f_2 \right\rangle\neq0$. The temporal autocovariance and autocorrelation along functions $f_1$ and $f_2$ provide practical observables as such quantities can be extracted from time series and real-life data. As an example, the choice of $f_1$ and $f_2$ as indicator functions implies the correlation and autocorrelation of the average of some of the components in $u$ on the supports of the observed functions \cite{bernuzzi2023early}. A concrete example of this approach is described in Section~\ref{ssec:Bous}. For any fixed $p\leq\lambda$ and under the assumption of completeness of the generalized eigenfunctions of $A_0(p)$, the functions $f_1$ and $f_2$ can also be assumed to be in $O_M(p)^*$, for a sufficiently large $M\in\mathbb{N}_{>0}$.

\subsection{A Boussinesq model with noise boundary conditions}\label{ssec:Bous}
The Boussinesq model, studied in \cite{dijkstra1997symmetry}, describes flows in a two-dimensional region of the ocean. Such area is defined by the spatial variables $(x_1,x_2)\in[-H,0]\times[0,L]$, for depth $H$ and latitude length $L$. The scaled and non-dimensionalized variables that define the model are the salinity $S$, the temperature $T$, the vorticity $\omega$ and the streamfunction $\psi$. The two-dimensional model is described as follows:
\begin{align} \label{bous}
    Pr^{-1} \left( \frac{\partial \omega}{\partial t} + u \frac{\partial \omega}{\partial x_2} + w \frac{\partial \omega}{\partial x_1} \right) &= \Delta \omega + Ra \left( \frac{\partial T}{\partial x_2} - \frac{\partial S}{\partial x_2} \right) \;,\nonumber\\
    \omega = - \Delta \psi \qquad , \qquad u &= \frac{\partial \psi}{\partial x_1} \qquad , \qquad w = -\frac{\partial \psi}{\partial x_2} \;,\nonumber\\
    \frac{\partial T}{\partial t} + u\frac{\partial T}{\partial x_2} + w\frac{\partial T}{\partial x_1} &= \Delta T \;,\\
    \frac{\partial S}{\partial t} + u\frac{\partial S}{\partial x_2} + w\frac{\partial S}{\partial x_1} &= Le^{-1}\Delta \nonumber S \;;
\end{align}
lateral boundary conditions for any $x_1\in(-H,0)$ and $t\geq0$ as
\begin{align}
    \label{bound1}
    x_2=0,L\;&: \qquad \psi(x_1,x_2,t)=\omega(x_1,x_2,t)=\frac{\partial S}{\partial x_2}(x_1,x_2,t)=\frac{\partial T}{\partial x_2}(x_1,x_2,t)=0\;;
\end{align}
conditions at the ocean floor for any $x_2\in(0,L)$ and $t\geq0$ as
\begin{align}
    \label{bound2}
    x_1=-H\;&: \qquad \psi(x_1,x_2,t)=\omega(x_1,x_2,t)=\frac{\partial S}{\partial x_1}(x_1,x_2,t)=\frac{\partial T}{\partial x_1}(x_1,x_2,t)=0\;;
\end{align}
conditions at the surface for any $x_2\in(0,L)$ and $t\geq0$ as
\begin{align}
    \label{bound3}
    &\psi(x_1,x_2,t)=\omega(x_1,x_2,t)=0\;, \nonumber\\
    x_1=0\;: \qquad &\frac{\partial S}{\partial x_1}(x_1,x_2,t)=p \left(Q_S(x_2) + \nu V_S(x_2) \right)+\sigma Q_S(x_2) \dot{W}(x_2,t)\;, \\
    &\frac{\partial T}{\partial x_1}(x_1,x_2,t)=-\kappa(T(x_1,x_2,t)-T_S(x_2)+\delta)\;.\nonumber
\end{align}
The initial conditions of the model are denoted as
\begin{equation*}
    u_0(x_1,x_2)=\left( \psi(x_1,x_2,0),\omega(x_1,x_2,0),T(x_1,x_2,0),S(x_1,x_2,0)\right)\;,
\end{equation*}
for any $x_1\in[-H,0]$ and $x_2\in[0,L]$. The term $\dot{W}$ indicates Gaussian white noise, in space and time, on the surface boundary. In the following SPDEs, the noise intensity is set as $\sigma=0.01$, unless stated otherwise. The Prandtl number $Pr$ and the Lewis number $Le$ are fixed as $Pr=2.25$ and $Le=1$. Following \cite{dijkstra1997symmetry}, we set $H=1$. We use $\Delta$ to denote the two-dimensional Laplace operator on $[-H,0]\times[0,L]$. For clarity, Figure \ref{fig:shape_bounds} displays the model set-up. The forcing functions are
\begin{align*}
    Q_S(x_2)&=3 \text{cos} \left( 2 \pi \left(\frac{x_2}{L}-\frac{1}{2}\right) \right) \;,\\
    V_S(x_2)&= -\text{sin}\left( \pi \left(\frac{x_2}{L} - \frac{1}{2} \right) \right),
\end{align*}
and
\begin{equation*}
    T_S(x_2)=\frac{1}{2} \left( \text{cos} \left( 2 \pi \left( \frac{x_2}{L}-\frac{1}{2} \right) \right) +1 \right)
\end{equation*}
for $x_2\in[0,L]$. The shape of functions $Q_S$, $V_S$ and $T_S$ are shown in Figure \ref{fig:shape_bounds}.

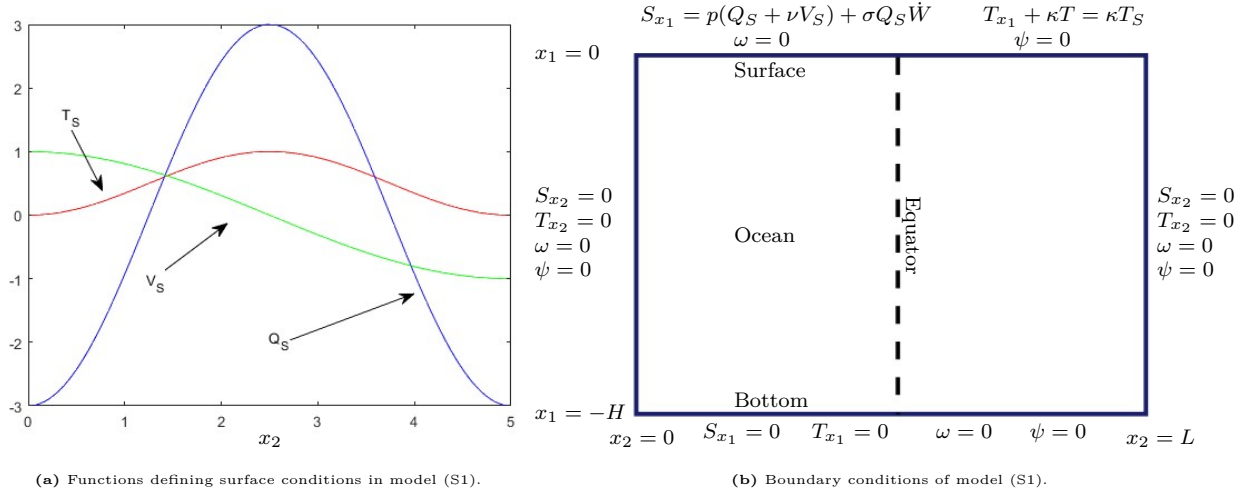
\begin{figure}[h!]
    \centering   
    \subfloat[Functions defining surface conditions in model (S1).]{\begin{overpic}[width= 0.45\textwidth]{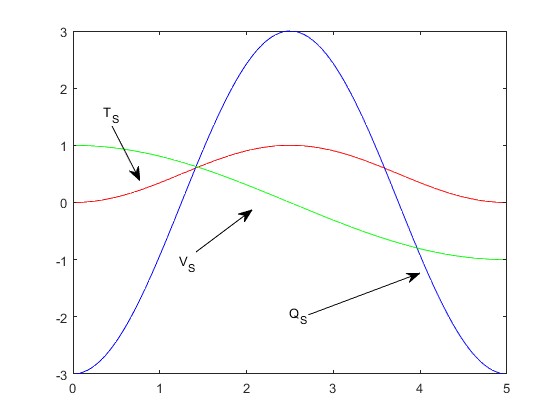}
    \put(500,20){\footnotesize{$x_2$}}
    \end{overpic}}
    \hspace{0mm}
    \subfloat[Boundary conditions of model (S1).]{\begin{overpic}[width= 0.45\textwidth]{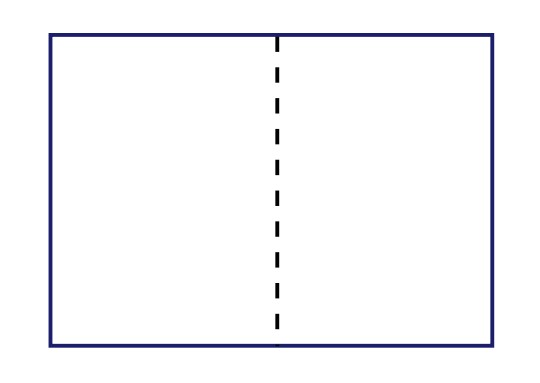}
                \put(250,610){\footnotesize{Surface}}
                \put(250,345){\footnotesize{Ocean}}
                \put(250,80){\footnotesize{Bottom}}
                \put(520,420){\rotatebox{270}{\footnotesize{Equator}}}
                \put(-70,640){\footnotesize{$x_1=0$}}
                \put(-70,60){\footnotesize{$x_1=-H$}}
                \put(50,20){\footnotesize{$x_2=0$}}
                \put(880,20){\footnotesize{$x_2=L$}}
                \put(-70,410){\footnotesize{$S_{x_2}=0$}}
                \put(-70,370){\footnotesize{$T_{x_2}=0$}}
                \put(-70,330){\footnotesize{$\omega=0$}}
                \put(-70,290){\footnotesize{$\psi=0$}}
                \put(930,410){\footnotesize{$S_{x_2}=0$}}
                \put(930,370){\footnotesize{$T_{x_2}=0$}}
                \put(930,330){\footnotesize{$\omega=0$}}
                \put(930,290){\footnotesize{$\psi=0$}}
                \put(200,30){\footnotesize{$S_{x_1}=0$}}
                \put(375,30){\footnotesize{$T_{x_1}=0$}}
                \put(575,30){\footnotesize{$\omega=0$}}
                \put(725,30){\footnotesize{$\psi=0$}}
                \put(100,700){\footnotesize{$S_{x_1}=p (Q_S+\nu V_S)+\sigma Q_S \dot{W}$}}
                \put(650,700){\footnotesize{$T_{x_1}+\kappa T=\kappa T_S$}}
                \put(250,660){\footnotesize{$\omega=0$}}
                \put(700,660){\footnotesize{$\psi=0$}}
    \end{overpic}}
    
    \caption{Panel (a) shows the shape of the forcing functions as functions of $x_2$ for $L=5$. The choice of function $T_S$ affects the temperature forcing, and the freshwater flux forcing is related to $Q_S$ and $V_S$. The function $V_S$ provides asymmetry in $x_2$ on the boundary conditions.\\
    The Boussinesq model (S1) setup and its boundary conditions are displayed in panel (b). The partial derivatives are indicated as $S_{x_1},T_{x_1},S_{x_2},T_{x_2}$. The boundary condition for the temperature at the surface is a Newtonian cooling condition, or Robin condition. Following \cite{baars2017continuation}, we insert white noise on the surface boundary condition of the salinity variable.}
    \label{fig:shape_bounds}
\end{figure}

We set a bifurcation threshold $\lambda$. For fixed $p<\lambda$ we pick a deterministically stable equilibrium solution $(\psi_{\ast},\omega_{\ast},T_{\ast},S_{\ast})$ of \eqref{bous}, \eqref{bound1}, \eqref{bound2} and \eqref{bound3} with $\sigma=0$. We can then linearize the system locally, thus obtaining
\begin{align} \label{bous_lin_syst}
    \begin{pmatrix}
    0 \\\\ \frac{\partial \omega}{\partial t} \\\\ \frac{\partial T}{\partial t} \\\\ \frac{\partial S}{\partial t}
    \end{pmatrix}
    = A(\psi_{\ast},\omega_{\ast},T_{\ast},S_{\ast})
    \begin{pmatrix}
    \psi \\\\ \omega \\\\ T \\\\ S
    \end{pmatrix}
\end{align}
and linearized boundary conditions \eqref{bound1}, \eqref{bound2} and
\begin{align}
    \label{new_bound3}
    &\psi(x_1,x_2,t)=\omega(x_1,x_2,t)=0\;, \nonumber\\
    x_1=0\;: \qquad &\frac{\partial S}{\partial x_1}(x_1,x_2,t)= \sigma Q_S(x_2) \dot{W}(x_2,t)\;, \\
    &\frac{\partial T}{\partial x_1}(x_1,x_2,t)+\kappa T(x_1,x_2,t)=0\;.\nonumber
\end{align}
for any $x_2\in(0,L)$ and $t\geq0$. The original system \eqref{bous} with boundary conditions \eqref{bound1}, \eqref{bound2} and \eqref{bound3} is labeled (S0) under the assumption $\sigma=0$ and (S1) when we use as an initial condition $u_0=(\psi_{\ast},\omega_{\ast},T_{\ast},S_{\ast})$. The linearized system \eqref{bous_lin_syst} with boundary conditions \eqref{bound1}, \eqref{bound2}, \eqref{new_bound3} and initial condition in the null function $u_0\equiv0$ is denoted as (S2).

The operator $A(\psi_{\ast},\omega_{\ast},T_{\ast},S_{\ast})$ is of the form
\begin{align*}
    A(\psi_{\ast},\omega_{\ast},T_{\ast},S_{\ast})
    =\left(\begin{array}{c|c}
    A_{11} & A_{12}\\
    \hline
    A_{21}(\omega_{\ast},T_{\ast},S_{\ast}) & A_{22}(\psi_{\ast})
    \end{array}\right)
    \;,
\end{align*}
for: $A_{11}=\Delta$, the Laplace operator;
\begin{align*}
    A_{12} 
    = \left(\begin{array}{c c c}
        \operatorname{I} & 0 & 0
    \end{array} \right)\;,
\end{align*}
with $\operatorname{I}$, the identity operator, and $0$, the null operator;
\begin{align*}
    A_{21}(\omega_{\ast},T_{\ast},S_{\ast})
    = \begin{pmatrix}
    -\frac{\partial \omega_{\ast}}{\partial x_2} \frac{\partial}{\partial x_1} + \frac{\partial \omega_{\ast}}{\partial x_1} \frac{\partial}{\partial x_2}\\\\
    -\frac{\partial T_{\ast}}{\partial x_2} \frac{\partial}{\partial x_1} + \frac{\partial T_{\ast}}{\partial x_1} \frac{\partial}{\partial x_2}\\\\
    -\frac{\partial S_{\ast}}{\partial x_2} \frac{\partial}{\partial x_1} + \frac{\partial S_{\ast}}{\partial x_1} \frac{\partial}{\partial x_2}
    \end{pmatrix} \;;
\end{align*}
for the operator
\begin{align*}
    A_{22}(\psi_{\ast})
    &=\begin{pmatrix}
    -\frac{\partial \psi_{\ast}}{\partial x_1} \frac{\partial}{\partial x_2} + \frac{\partial \psi_{\ast}}{\partial x_2} \frac{\partial}{\partial x_1} + Pr \Delta
    & Pr \; Ra \frac{\partial}{\partial x_2}
    & - Pr \; Ra \frac{\partial}{\partial x_2} \\
    & & \\
    0
    & -\frac{\partial \psi_{\ast}}{\partial x_1} \frac{\partial}{\partial x_2} + \frac{\partial \psi_{\ast}}{\partial x_2} \frac{\partial}{\partial x_1} + \Delta
    & 0 \\
    & & \\
    0
    & 0
    & -\frac{\partial \psi_{\ast}}{\partial x_1} \frac{\partial}{\partial x_2} + \frac{\partial \psi_{\ast}}{\partial x_2} \frac{\partial}{\partial x_1} + Le^{-1} \Delta
    \end{pmatrix} \;.
\end{align*}
Solutions $(\psi,\omega,T,S)$ of \eqref{bous_lin_syst} satisfy
\begin{align*}
    \begin{cases}
    -A_{11} \psi = A_{12} \begin{pmatrix}
    \omega\\T\\S
    \end{pmatrix} \;, \\
    \frac{\partial}{\partial t} \begin{pmatrix}
    \omega\\T\\S
    \end{pmatrix} = A_{21}(\omega_{\ast},T_{\ast},S_{\ast}) \psi + A_{22}(\psi_{\ast}) \begin{pmatrix}
    \omega\\T\\S
    \end{pmatrix} \;.
    \end{cases}
\end{align*}
Under boundary conditions \eqref{bound1}, \eqref{bound2} and \eqref{new_bound3}, $A_{11}$ is invertible and the Schur complement of $A(\psi_{\ast},\omega_{\ast},T_{\ast},S_{\ast})$ is defined as
\begin{align} \label{eq:def_AS}
    &A_\text{S}(\psi_{\ast},\omega_{\ast},T_{\ast},S_{\ast}):=A_{22}(\psi_{\ast})-A_{21}(\omega_{\ast},T_{\ast},S_{\ast}) A_{11}^{-1} A_{12} \\
    &=\begin{pmatrix}
    -\frac{\partial \psi_{\ast}}{\partial x_1} \frac{\partial}{\partial x_2} + \frac{\partial \psi_{\ast}}{\partial x_2} \frac{\partial}{\partial x_1} + Pr \Delta + \left( \frac{\partial \omega_{\ast}}{\partial x_2} \frac{\partial}{\partial x_1}  - \frac{\partial \omega_{\ast}}{\partial x_1} \frac{\partial}{\partial x_2} \right) \Delta^{-1}
    & Pr \; Ra \frac{\partial}{\partial x_2}
    & - Pr \; Ra \frac{\partial}{\partial x_2} \\
    & & \\
    \left(\frac{\partial T_{\ast}}{\partial x_2} \frac{\partial}{\partial x_1} -  \frac{\partial T_{\ast}}{\partial x_1} \frac{\partial}{\partial x_2} \right) \Delta^{-1}
    & -\frac{\partial \psi_{\ast}}{\partial x_1} \frac{\partial}{\partial x_2} + \frac{\partial \psi_{\ast}}{\partial x_2} \frac{\partial}{\partial x_1} + \Delta
    & 0 \\
    & & & \\
    \left(\frac{\partial S_{\ast}}{\partial x_2} \frac{\partial}{\partial x_1} - \frac{\partial S_{\ast}}{\partial x_1} \frac{\partial}{\partial x_2} \right) \Delta^{-1}
    & 0
    & -\frac{\partial \psi_{\ast}}{\partial x_1} \frac{\partial}{\partial x_2} + \frac{\partial \psi_{\ast}}{\partial x_2} \frac{\partial}{\partial x_1} + Le^{-1} \Delta
    \end{pmatrix} \;. \nonumber
\end{align}
We introduce also 
\begin{equation} \label{eq:A0_1}
    A_0(\psi_{\ast},\omega_{\ast},T_{\ast},S_{\ast})=A_\text{S}(\psi_{\ast},\omega_{\ast},T_{\ast},S_{\ast})
\end{equation} 
for 
\begin{equation} \label{eq:A0_2}
    \mathcal{D}\left(A_0(\psi_{\ast},\omega_{\ast},T_{\ast},S_{\ast})\right)
    =\mathcal{D}\left(A_S(\psi_{\ast},\omega_{\ast},T_{\ast},S_{\ast})\right)\cap
    \left\{\left(\omega,T,S\right) \quad \text{that satisfy \eqref{bound1}, \eqref{bound2} and \eqref{new_bound3} for $\sigma=0$}\; \right\}\;,
\end{equation}
and any steady solution $(\psi_{\ast},\omega_{\ast},T_{\ast},S_{\ast})$ of (S0).
We can then consider the linearized problem 
\begin{align*}
    \frac{\partial}{\partial t} \begin{pmatrix}
    \omega\\T\\S
    \end{pmatrix} = A_\text{S}(\psi_{\ast},\omega_{\ast},T_{\ast},S_{\ast}) \begin{pmatrix}
    \omega\\T\\S
    \end{pmatrix}
\end{align*}
with boundary conditions \eqref{bound1}, \eqref{bound2} and \eqref{new_bound3}. Setting an initial condition in $\mathcal{D}(A_\text{S}(\psi_{\ast},\omega_{\ast},T_{\ast},S_{\ast}))$, the system is of the form \eqref{eq:fast_linear}. In particular, we can now apply Theorem \ref{thm:general}, Theorem \ref{thm:second} and Corollary \ref{cor:autocorrelation}. The results are discussed and compared to numerical methods in the examples to follow.\medskip

The variable $S$ appears in the model (S1) only under the application of derivative operators, thus indicating the invariance of the time derivatives and boundary conditions under shifting by a constant along the $S$ component. For all stable solutions of system (S0), $(\psi_{\ast},\omega_{\ast},T_{\ast},S_{\ast})$, the eigenvalue with highest real part of $A_0(\psi_{\ast},\omega_{\ast},T_{\ast},S_{\ast})$ is $\lambda_1^{(p)}=0$ for any $p<\lambda$. Its corresponding eigenfunction $e_{1,1}^{(p)}$ is characterized by
\begin{align} \label{eq:e_11}
    e_{1,1}^{(p)}(x_1,x_2)=
    \begin{pmatrix}
        0\\0\\c
    \end{pmatrix}\;,
\end{align}
for $c\in\mathbb{R}$, for any $p\leq\lambda$ and $(x_1,x_2)\in[-H,0]\times[0,L]$. In the following examples, we ensure that the early-warning sign is not affected by the presence of an eigenvalue with zero real part.

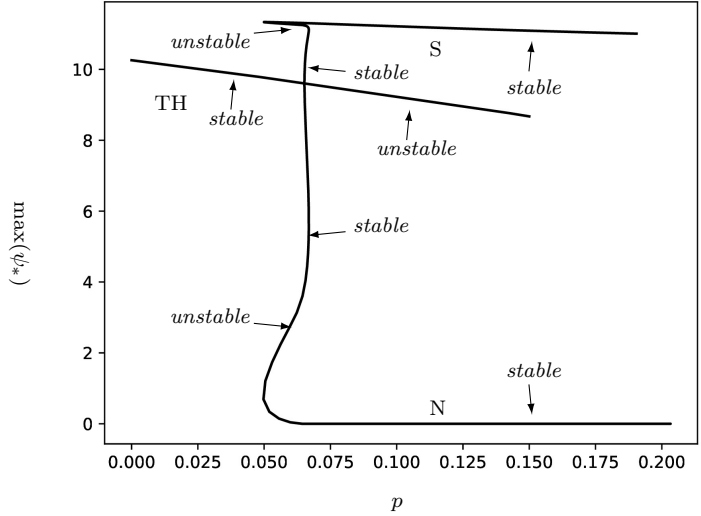
\begin{figure}[h!]
    \centering
    \subfloat{
        \begin{overpic}[scale=0.15]{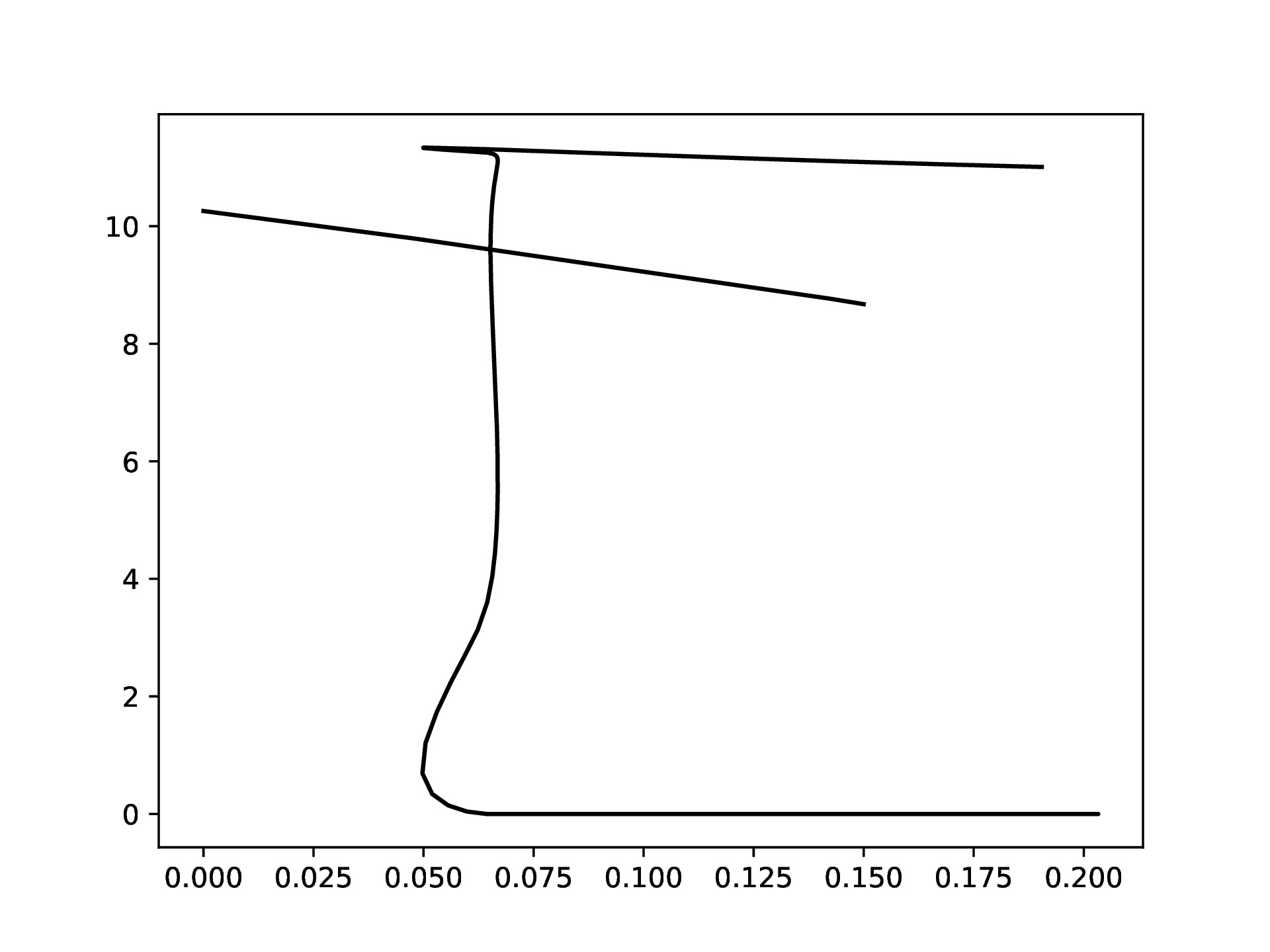}
        \put(190,520){\footnotesize{TH}}
        \put(550,590){\footnotesize{S}}
        \put(550,120){\footnotesize{N}}
        \put(650,170){\footnotesize{\textit{stable}}}
        \put(680,160){\vector(0.1,-1){4}}
        \put(650,540){\footnotesize{\textit{stable}}}
        \put(680,570){\vector(0.1,1){4}}
        \put(260,500){\footnotesize{\textit{stable}}}
        \put(290,520){\vector(0.1,1){4}}
        \put(450,560){\footnotesize{\textit{stable}}}
        \put(440,570){\vector(-1,0.1){50}}
        \put(450,360){\footnotesize{\textit{stable}}}
        \put(440,360){\vector(-1,-0.1){50}}
        \put(480,460){\footnotesize{\textit{unstable}}}
        \put(520,480){\vector(0.1,1){4}}
        \put(210,600){\footnotesize{\textit{unstable}}}
        \put(320,620){\vector(1,0.1){50}}
        \put(210,240){\footnotesize{\textit{unstable}}}
        \put(320,240){\vector(1,-0.1){50}}
        \put(0,400){\rotatebox{270}{\footnotesize{max$\left(\psi_\ast\right)$}}}
        \put(500,0){\footnotesize{$p$}}
        \end{overpic}}
    \caption{Section of the bifurcation diagram of (S0) under the assumptions $Ra=10^4$, $\kappa=10^2$, $L=10$ and $\nu=\delta=0$, for the parameter $p$. The thermally dominated 2-cell state, TH-solution, is stable for $p<\lambda\approx 0.0638$ and unstable for $p>\lambda$. The value $\lambda$ indicates a supercritical pitchfork threshold. The system is yet multistable for $p>\tilde{\lambda}\approx0.049$, as two saddle-node bifurcations indicate the presence of the southward sinking solution, S-solution, and the northward sinking solution, N-solution.\\
    The diagram is obtained with the Python library Transiflow \cite{transiflow} through a pseudo-arclength continuation method and uniform resolution grid $(\mathtt{M},\mathtt{N})=(30,100)$, similar as \cite{dijkstra1997symmetry}.}
    \label{fig:bif_pitchfork}
\end{figure}

\begin{example}
In this example, we assume that $L=10$, $\nu=\delta=0$, and $Ra=10^4$. The value $\kappa$, a ratio of the diffusive time scale of vertical heat transfer and the relaxation time scale, is assumed as $\kappa=100$. The bifurcation diagram of system (S0) is displayed in Figure \ref{fig:bif_pitchfork}.  It is shown that a supercritical pitchfork bifurcation threshold is present at $p=\lambda\approx0.0638$, where a stable solution with components $T$ and $S$ symmetric along the equator $x_2=\frac{L}{2}$, the thermally dominated 2-cell state, loses its stability. Such solution is labeled as $(\psi_\text{TH},\omega_\text{TH},T_\text{TH},S_\text{TH})$, chosen as $(\psi_{\ast},\omega_{\ast},T_{\ast},S_{\ast})$ and shown in the first row of Figure \ref{fig:solution_fixed} for $p=0.055$. At $p=\tilde{\lambda}\approx0.049$, two pairs of solutions appear from saddle-node bifurcations, inducing multistability in the system for $\tilde{\lambda}<p\leq\lambda$.  One of the stable solutions that arises from such threshold, the southward sinking solution, is shown in the second row of Figure \ref{fig:solution_fixed}. The northward sinking solution can be constructed from the latter. Proof of such a statement is provided in Appendix \ref{app:math_appendix}.

\begin{figure}[h!]
    \centering
    \subfloat[Streamfunction $\psi_\text{TH}$]{\begin{overpic}[scale=0.22]{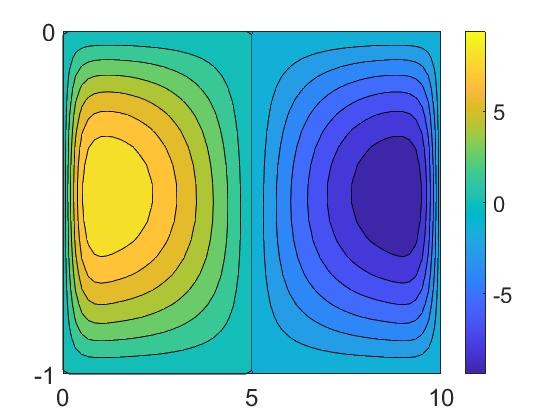}
    \put(-20,380){\rotatebox{270}{\footnotesize{$x_1$}}}
    \put(530,-20){\footnotesize{$x_2$}}
    \end{overpic}}
    \hspace{0mm}
    \subfloat[Vorticity $\omega_\text{TH}$]{\begin{overpic}[scale=0.22]{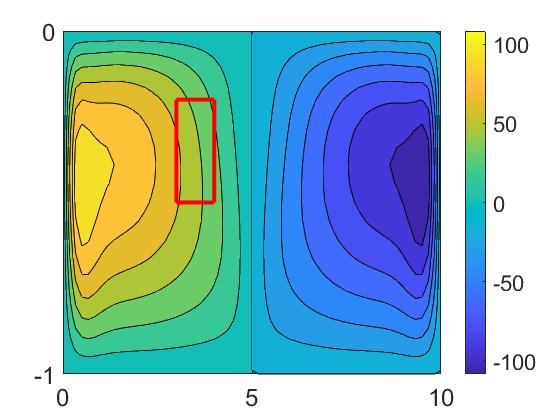}
    \put(-20,380){\rotatebox{270}{\footnotesize{$x_1$}}}
    \put(530,-20){\footnotesize{$x_2$}}
    \end{overpic}}
    \hspace{0mm}
    \subfloat[Temperature $T_\text{TH}$]{\begin{overpic}[scale=0.22]{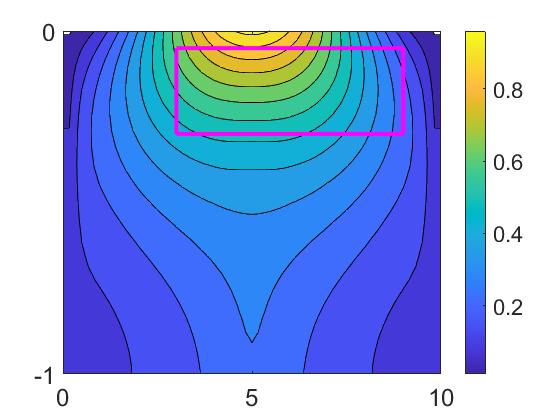}
    \put(-20,380){\rotatebox{270}{\footnotesize{$x_1$}}}
    \put(530,-20){\footnotesize{$x_2$}}
    \end{overpic}}
    \hspace{0mm}
    \subfloat[Salinity $S_\text{TH}$]{\begin{overpic}[scale=0.22]{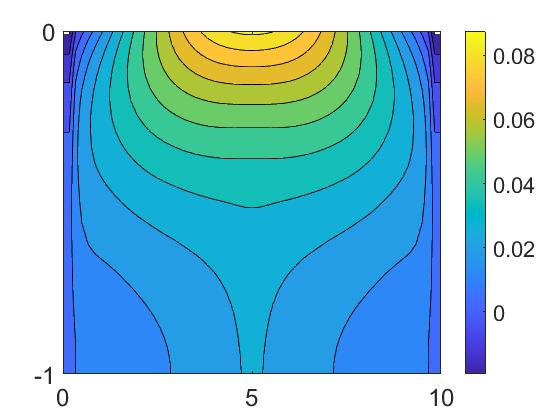}
    \put(-20,380){\rotatebox{270}{\footnotesize{$x_1$}}}
    \put(530,-20){\footnotesize{$x_2$}}
    \end{overpic}}
    
    \vspace{0mm}
    
    \subfloat[Streamfunction $\psi_\text{S}$]{\begin{overpic}[scale=0.22]{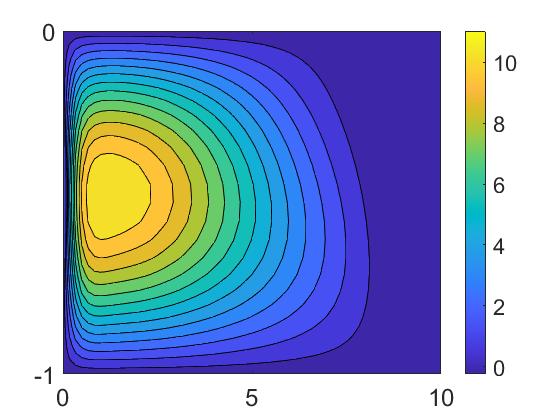}
    \put(-20,380){\rotatebox{270}{\footnotesize{$x_1$}}}
    \put(530,-20){\footnotesize{$x_2$}}
    \end{overpic}}
    \hspace{0mm}
    \subfloat[Vorticity $\omega_\text{S}$]{\begin{overpic}[scale=0.22]{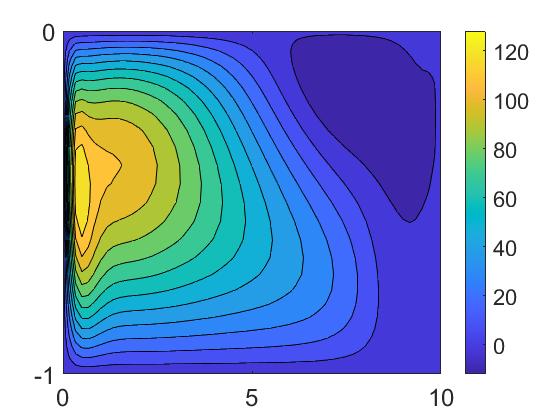}
    \put(-20,380){\rotatebox{270}{\footnotesize{$x_1$}}}
    \put(530,-20){\footnotesize{$x_2$}}
    \end{overpic}}
    \hspace{0mm}
    \subfloat[Temperature $T_\text{S}$]{\begin{overpic}[scale=0.22]{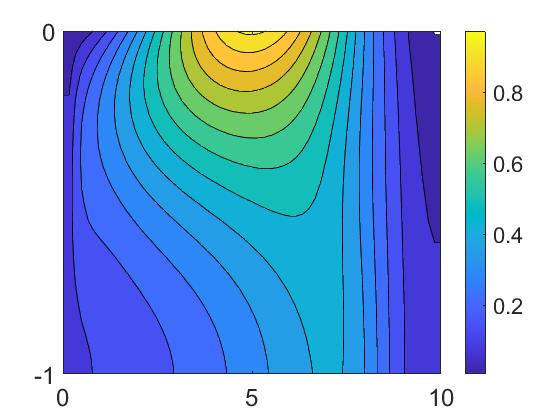}
    \put(-20,380){\rotatebox{270}{\footnotesize{$x_1$}}}
    \put(530,-20){\footnotesize{$x_2$}}
    \end{overpic}}
    \hspace{0mm}
    \subfloat[Salinity $S_\text{S}$]{\begin{overpic}[scale=0.22]{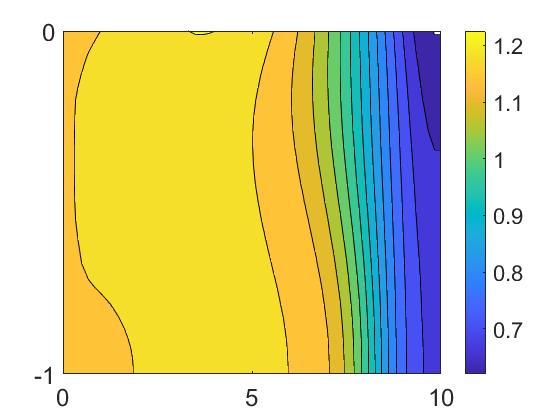}
    \put(-20,380){\rotatebox{270}{\footnotesize{$x_1$}}}
    \put(530,-20){\footnotesize{$x_2$}}
    \end{overpic}}
        \caption{The first row of contour plots presents the simulation of a stable solution $(\psi_\text{TH},\omega_\text{TH},T_\text{TH},S_\text{TH})$ of the system defined by (S0) and assumptions $Ra=10^4$, $\kappa=10^2$, $L=10$ and $\nu=\delta=0$. The parameter leading to a pitchfork bifurcation is taken as $p=0.055$, hence less than the threshold value $\lambda$. The coloured boxes delimit supports of indicator functions, chosen further in the example as $f_1^{(p)}$ and $f_2^{(p)}$, for any $p<\lambda$, in \eqref{eq:main_order_cor} in Theorem \ref{thm:second}. The red box, in the plot of $\omega_\text{TH}$, defines the rectangle $[-0.5,-0.2]\times[3,4]\subset[-H,0]\times[0,L]$ on the domain of $\omega_\text{TH}$. Such shape delimits the support of the indicator function $\mathbbm{1}_{\omega}$. Similarly, the magenta box, in the plot of $T_\text{TH}$, defines the rectangle $[-0.3,-0.05]\times[3,9]\subset[-H,0]\times[0,L]$ on the domain of $T_\text{TH}$, which indicates the support of the indicator function $\mathbbm{1}_{T}$.\\
        The second row presents the southward sinking steady solution $(\psi_\text{S},\omega_\text{S},T_\text{S},S_\text{S})$ at $p=0.055$, under equivalent assumptions. Such a stable solution arises from a saddle-node bifurcation at $p=\tilde{\lambda}\approx0.049$. }
    \label{fig:solution_fixed}
\end{figure}
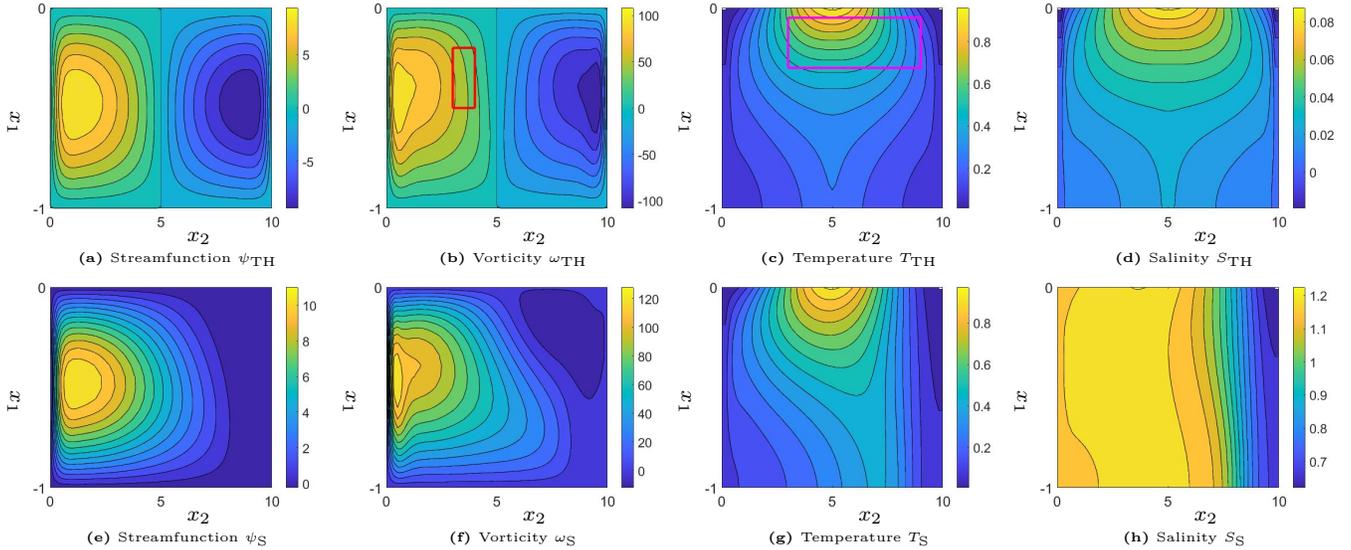 

The solutions of the original system (S1) and of the linearized system (S2) are obtained through an implicit Euler finite difference method with time step $10^{-2}$ and final time $t_\text{end}$. Similarly to \cite{dijkstra1997symmetry}, $\mathtt{M}=29$ and $\mathtt{N}=59$ internal resolution points have been considered for the intervals $(-H,0)$ and $(0,L)$ respectively. The grid is built by assuming the resolution coordinates $\left\{\mathtt{x}_i\right\}_{i\in\left\{1,\dots,\mathtt{N}\right\}}=\left\{\frac{i\;A}{\mathtt{N}+1}\right\}_{i\in\left\{1,\dots,\mathtt{N}\right\}}$, along the $x_2$ direction, and $\left\{\mathtt{z}_i\right\}_{i\in\left\{1,\dots,\mathtt{M}\right\}}$, along the $x_1$ direction, for
\begin{equation*}
    \mathtt{z}_i:=-\frac{1}{2}-\frac{\text{tanh}\left(-3\left(\mathtt{y}_i+\frac{1}{2}\right)\right)}{2\text{tanh}\left(\frac{3}{2}\right)}
\end{equation*}
and $\left\{\mathtt{y}_i\right\}_{i\in\left\{1,\dots,\mathtt{M}\right\}}=\left\{-1+\frac{i}{\mathtt{M}+1}\right\}_{i\in\left\{1,\dots,\mathtt{M}\right\}}$. Such choice of the grid, similar to the one described in \cite{dijkstra1992structure}, is selected to achieve higher resolution in the proximity to the surface and ocean floor boundary. The numerical techniques employed to obtain first- and second-order derivatives in a non-equidistant grid are described in \cite{veldman1992playing}.

\begin{figure}[h!]
    \centering
    \subfloat{\begin{overpic}[scale=0.5]{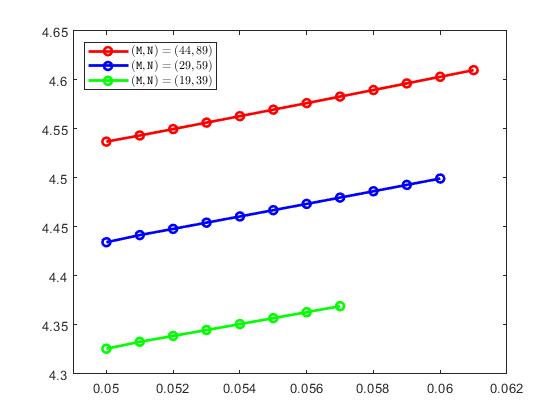}
    \put(500,10){\footnotesize{$p$}}
    \put(10,450){\rotatebox{270}{\footnotesize{$\lambda_2^{(p)}-\lambda_3^{(p)}$}}}
    \end{overpic}}
    \caption{Plot of $\lambda_2^{(p)}-\lambda_3^{(p)}$, the difference between the second and third eigenvalues of $A_0(\psi_\text{TH},\omega_\text{TH},T_\text{TH},S_\text{TH})$ with highest real part, for various resolution grid choices. In green the resolution is assumed to be $(\mathtt{M},\mathtt{N})=(19,39)$, in blue $(\mathtt{M},\mathtt{N})=(29,59)$ and in red $(\mathtt{M},\mathtt{N})=(44,89)$. The differences appear to increase in the proximity of the bifurcation threshold, which assumes lower values for coarse resolution grids, as discussed in \cite{dijkstra1997symmetry}.}
    \label{fig:diff_eigs}
\end{figure}

We consider $(\psi_\text{TH},\omega_\text{TH},T_\text{TH},S_\text{TH})$ as a family of stable steady solutions of (S0) as $p$ approaches a supercritical pitchfork bifurcation at $\lambda$. The shape of the spectrum of the operator $A_0(\psi_\text{TH},\omega_\text{TH},T_\text{TH},S_\text{TH})$ is not known; in particular, its discreteness is not trivial to prove through analytic methods. Following Remark \ref{rem:generalization}, we study the behaviour of the spectrum region with the highest real part as $p$ approaches the bifurcation threshold. The threshold value is obtained numerically with more precision for increased grid resolution of the space. We show in Figure \ref{fig:diff_eigs} the difference between the second and third eigenvalues with highest real part of the numerical approximation, by finite difference method, of $A_0(\psi_\text{TH},\omega_\text{TH},T_\text{TH},S_\text{TH})$ previous to the registered bifurcation threshold, for different values of $\mathtt{N}$ and $\mathtt{M}$. The difference appears to be steadily distant from $0$, indicating a gap between such eigenvalues, and increasing as $p$ approaches the detected bifurcation point. We can, therefore, compare the results of Corollary \ref{cor:autocorrelation} with simulations obtained from the model.

\begin{figure}[h!]
    \centering
    
    \subfloat[Comparison of variance and expected scaling law under white noise.]{\begin{overpic}[scale=0.45]{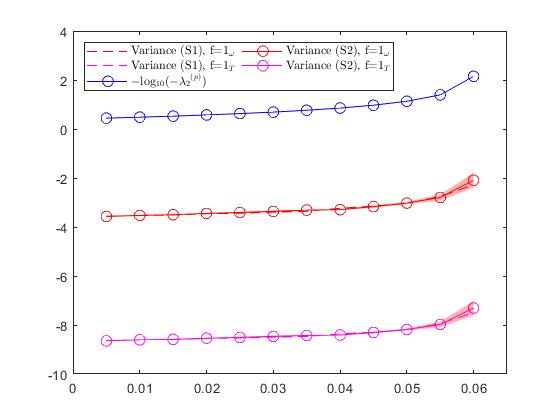}
        \put(500,10){\footnotesize{$p$}}
        \put(30,500){\rotatebox{270}{\footnotesize{$\text{log}_{10}\left(\left\langle f, V_{t_\text{end}} f \right\rangle\right)$}}}
    \end{overpic}
    \label{fig:5a}}
    \hfill
    \subfloat[Coefficients associated to the choice of the observables.]{\begin{overpic}[scale=0.45]{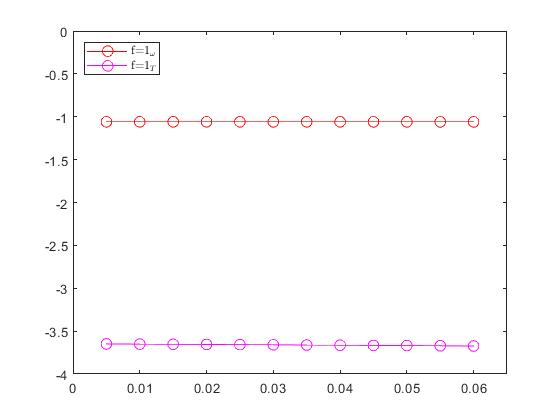}
        \put(500,10){\footnotesize{$p$}}
        \put(0,500){\rotatebox{270}{\footnotesize{$\text{log}_{10}\left(\left\langle f, e_{2,1}^{(p)} \right\rangle\right)$}}}
    \end{overpic}
    \label{fig:5b}}
    \caption{The blue line shows the behaviour of $\text{log}_{10}\left(-\frac{1}{\lambda_2^{(p)}}\right)$, for $\lambda_2^{(p)}$ the second eigenvalue of $A_0(\psi_\text{TH},\omega_\text{TH},T_\text{TH},S_\text{TH})$ with highest real part.\\
    In Figure \ref{fig:5a}, the red and magenta lines refer to means of observables obtained from $5$ run samples, differing from noise realizations. Such lines display respectively the mean of the logarithm of the temporal autocovariance, for time ${t_\text{end}}=10^3$ and $\tau=0$, of solutions of a numerically approximated system projected on indicator functions. The red lines indicate projection on $f=\mathbbm{1}_{\omega}$ and the magenta lines on $f=\mathbbm{1}_{T}$. Such functions are defined in the caption of Figure \ref{fig:solution_fixed}. The solutions in Figure \ref{fig:5a} associated with the dashed lines satisfy (S1), and those related to the solid lines solve (S2).\\
    In Figure \ref{fig:5a}, the circles on the red and magenta solid lines indicate the mean values of observables from the linearized systems. The shaded areas have widths equal to double the corresponding standard deviations. The stable solutions are obtained through the natural parameter continuation method, and their stability is observed by solving the eigenvalue problem corresponding to the numerical approximation of $A_0(\psi_\text{TH},\omega_\text{TH},T_\text{TH},S_\text{TH})$ through the finite difference method. For $p>\tilde{\lambda}\approx0.049$, the sampled solutions of (S1) do not present a jump to a neighbourhood of another stable solution before ${t_\text{end}}$, which suggests that they remain in the basin of attraction of $(\psi_\text{TH},\omega_\text{TH},T_\text{TH},S_\text{TH})$ for such interval of time.\\
    Figure \ref{fig:5b} presents the scalar product in $\mathcal{H}$ of $f=\mathbbm{1}_{\omega}$, in red, and $f=\mathbbm{1}_{T}$, in magenta, with $e_{2,1}^{(p)}$. The values are not heavily dependent on $p$ and do not affect the observation of the rate in Figure \ref{fig:5a}.}
    \label{fig:rate_bous}
\end{figure}

In Figure \ref{fig:5a}, we observe in blue the behaviour of $\text{log}_{10}\left(-\frac{1}{\lambda_2^{(p)}}\right)$, for $\lambda_2^{(p)}$ being the second eigenvalue of $A_0(\psi_\text{TH},\omega_\text{TH},T_\text{TH},S_\text{TH})$ with largest real part. Depending on the grid resolution, such value increases as $p$ approaches the registered bifurcation threshold in $p\approx0.0618$. On the circles shown along the red and magenta solid lines in Figure \ref{fig:5a}, we observe the mean logarithm of the temporal variance, for time ${t_\text{end}}=10^3$, of $5$ solutions, differing from noise realizations, of the numerically approximated linearized system (S2), projected respectively on the indicator functions $\mathbbm{1}_{\omega}$ and $\mathbbm{1}_{T}$ with supports defined by the boxes in Figure \ref{fig:solution_fixed}. From ergodicity, implied by the additive noise term and the diffusion of the system, we can expect the time-asymptotic temporal variance of an observable of a simulation to be equal to the time-asymptotic variance in function space of the solution of the system along the function that defines the observable \cite{da1996ergodicity,KU2}. On the red and magenta dashed lines in Figure \ref{fig:5a}, we display the mean logarithm of the temporal variance in time ${t_\text{end}}$ for $5$ solutions, under different noise samples, of the numerically approximated original system (S1), projected on $\mathbbm{1}_{\omega}$ and $\mathbbm{1}_{T}$ respectively.

\begin{figure}[h!]
    \centering
    
    \subfloat[Vorticity $\omega$ in $e_{2,1}^{(0.055)}$]{\begin{overpic}[scale=0.22]{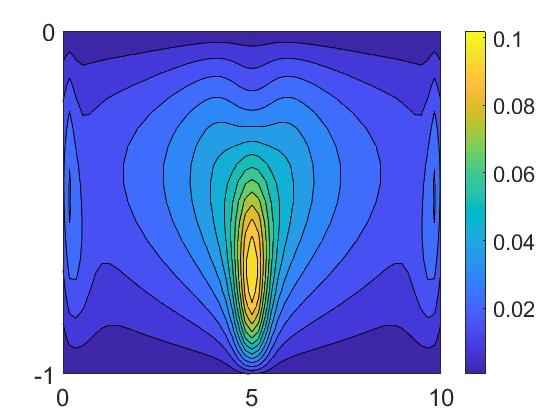}
    \put(-20,380){\rotatebox{270}{\footnotesize{$x_1$}}}
    \put(530,-20){\footnotesize{$x_2$}}
    \end{overpic}}
    \hspace{0mm}
    \subfloat[Temperature $T$ in $e_{2,1}^{(0.055)}$]{\begin{overpic}[scale=0.22]{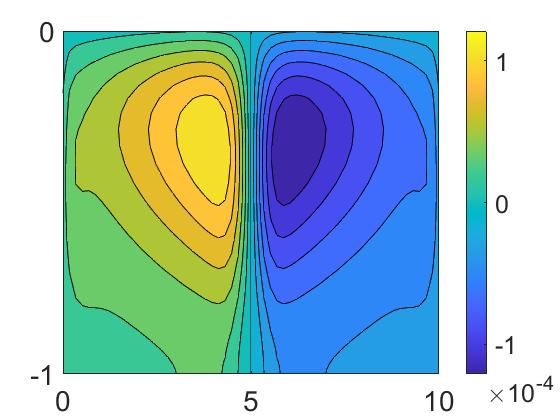}
    \put(-20,380){\rotatebox{270}{\footnotesize{$x_1$}}}
    \put(530,-20){\footnotesize{$x_2$}}
    \end{overpic}}
    \hspace{0mm}
    \subfloat[Salinity $S$ in $e_{2,1}^{(0.055)}$]{\begin{overpic}[scale=0.22]{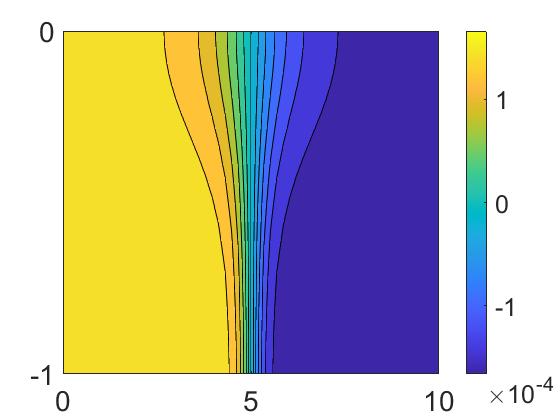}
    \put(-20,380){\rotatebox{270}{\footnotesize{$x_1$}}}
    \put(530,-20){\footnotesize{$x_2$}}
    \end{overpic}}
    
    \vspace{0mm}
    
    \subfloat[Vorticity $\omega$ in ${e_{2,1}^{(0.055)}}^*$]{\begin{overpic}[scale=0.22]{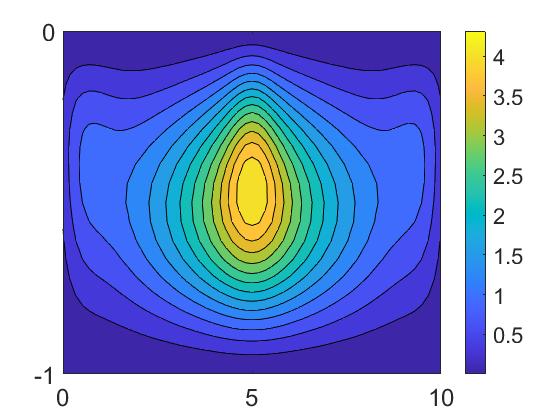}
    \put(-20,380){\rotatebox{270}{\footnotesize{$x_1$}}}
    \put(530,-20){\footnotesize{$x_2$}}
    \end{overpic}}
    \hspace{0mm}
    \subfloat[Temperature $T$ in ${e_{2,1}^{(0.055)}}^*$]{\begin{overpic}[scale=0.22]{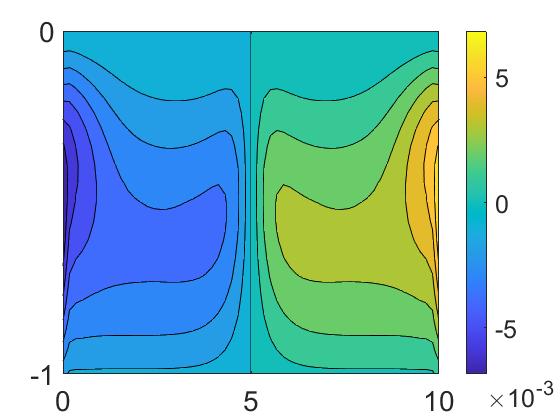}
    \put(-20,380){\rotatebox{270}{\footnotesize{$x_1$}}}
    \put(530,-20){\footnotesize{$x_2$}}
    \end{overpic}}
    \hspace{0mm}
    \subfloat[Salinity $S$ in ${e_{2,1}^{(0.055)}}^*$]{\begin{overpic}[scale=0.22]{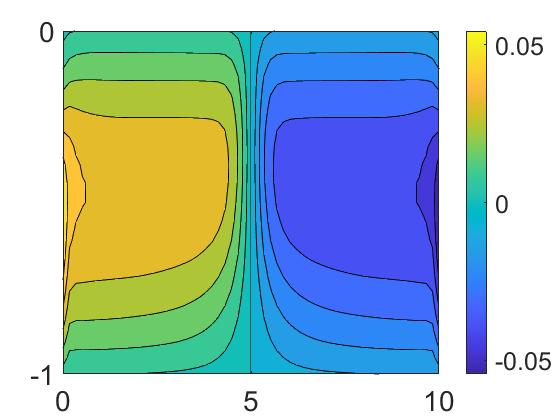}
    \put(-20,380){\rotatebox{270}{\footnotesize{$x_1$}}}
    \put(530,-20){\footnotesize{$x_2$}}
    \end{overpic}}
    \caption{The first and second row of contour plots display the components of the approximations of $e_{2,1}^{(0.055)}$ and ${e_{2,1}^{(0.055)}}^*$, which are respectively the (unique) eigenfunctions of $A_0(\psi_\text{TH},\omega_\text{TH},T_\text{TH},S_\text{TH})$ and $A_0(\psi_\text{TH},\omega_\text{TH},T_\text{TH},S_\text{TH})^*$ associated to the second eigenvalues with highest real part.}
    \label{fig:eigs}
\end{figure}
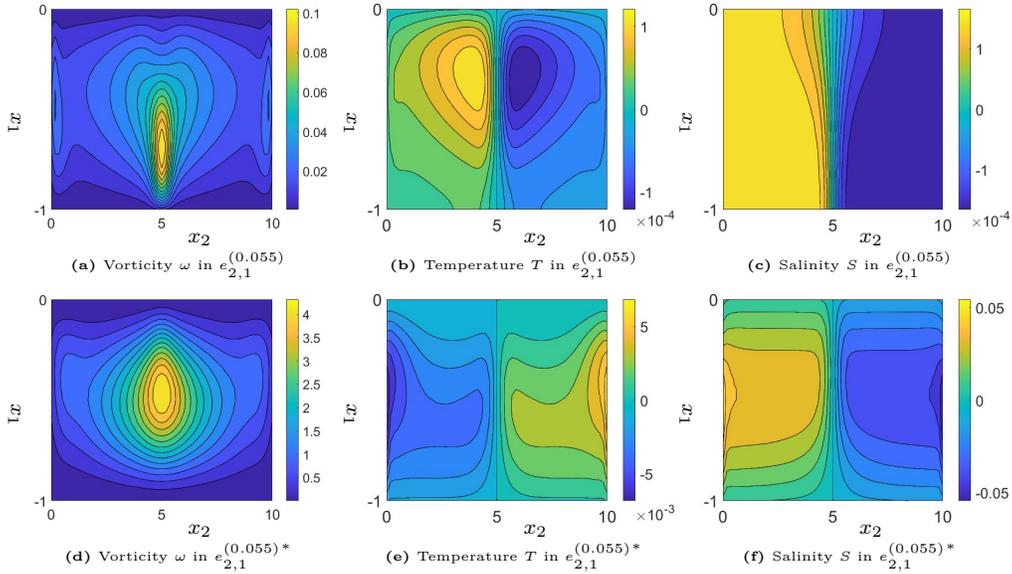 

The behaviour of the solid and dashed lines of the same colour is similar, as expected up to proximity to the bifurcation. They also resemble, as $p$ approaches the approximated bifurcation threshold, the blue line with the difference of a constant, suggesting an order of divergence as presented in \eqref{eq:main_order_cor}. Such resemblance is yet affected by the choices of functions $\mathbbm{1}_{\omega}$ and $\mathbbm{1}_{T}$. As described in Corollary \ref{cor:autocorrelation}, the time-asymptotic autocovariance with $\tau=0$, diverges for a dense set of functions in $\mathcal{H}$ as $\operatorname{Re}\left( -\lambda_1^{(p)}\right)^{-2 M_1 +1}$, but the chosen direction functions are orthogonal to the unique generalized eigenfunction, $e_{1,1}^{(p)}$, of $A_0(\psi_\text{TH},\omega_\text{TH},T_\text{TH},S_\text{TH})$ associated to $\lambda_1^{(p)}$, as described in \eqref{eq:e_11}. The leading term in the series \eqref{eq:second_EWS} is therefore associated to $\lambda_2^{(p)}$ and the corresponding generalized eigenfunctions of $A_0(\psi_\text{TH},\omega_\text{TH},T_\text{TH},S_\text{TH})$ and its adjoint. The rate displayed by the solid lines in Figure \ref{fig:5a} is, in fact, $\operatorname{Re}\left(- \lambda_2^{(p)}\right)^{-2 M_2 +1}$. In Figure \ref{fig:5b}, the scalar products in $\mathcal{H}$ of $\mathbbm{1}_{\omega}$ and $\mathbbm{1}_{T}$ with the second simple eigenfunction, $e_{2,1}^{(p)}$, of $A_0(\psi_\text{TH},\omega_\text{TH},T_\text{TH},S_\text{TH})$ are presented. Such value does not appear to be highly dependent on $p$; therefore, the choice of observables does not impact greatly the rate of the early-warning sign. For each observed value $p$, there exists only a simple eigenfunction related to $\lambda_2^{(p)}$, which is displayed in Figure \ref{fig:eigs} for $p=0.055$, and $M_2=1$.\\
The divergences in Figure \ref{fig:5a} can also be affected by

\begin{align*}
    \left\lvert\left\langle {e_{2,1}^{(p)}}^*,  D(p) B B^* D(p)^* {e_{2,1}^{(p)}}^* \right\rangle\right\rvert \;.
\end{align*}

The multistability of the system (S0) for $p$ close to the pitchfork bifurcation threshold $\lambda$ requires a careful choice of ${t_\text{end}}$. Such time needs to capture the time-asymptotic property of Lemma \ref{lm:lyap} and still has to be associated with a small probability of jumping from $(\psi_\text{TH},\omega_\text{TH},T_\text{TH},S_\text{TH})$ to another stable solution for any $p<\lambda$ observed. If the second property is not fulfilled, attracting properties in other basins of attraction can be observed, affecting the results in Figure \ref{fig:rate_bous}. The small intensity of the noise in the simulations of the solution of the system (S1) makes such an occurrence rare and not observed in the simulations for $\tilde{\lambda}< p<\lambda$.

\end{example}

\begin{figure}[h!]
    \centering
    \subfloat{
        \begin{overpic}[scale=0.15]{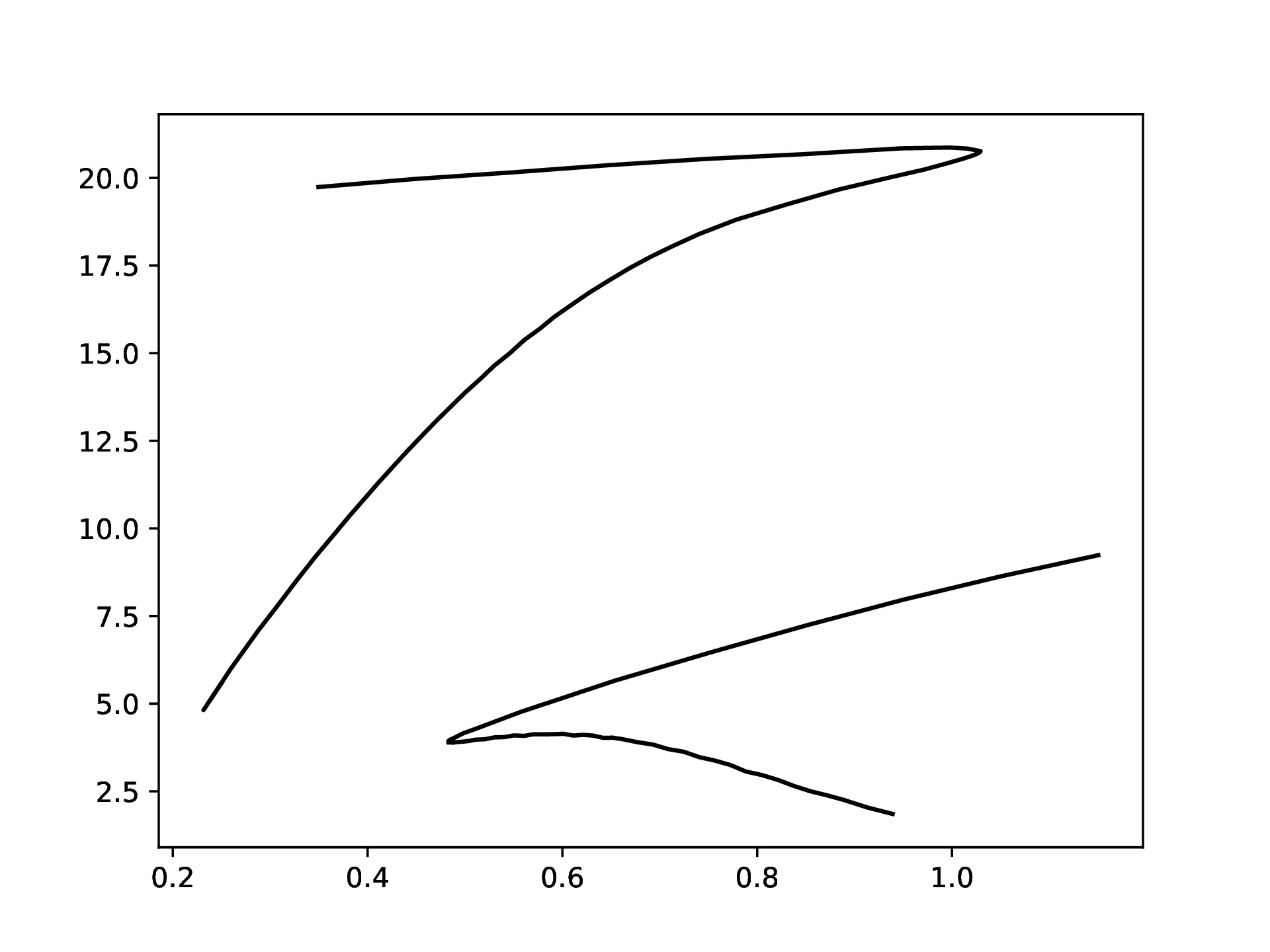}
        \put(490,590){\footnotesize{SS}}
        \put(490,250){\footnotesize{SA}}
        \put(200,550){\footnotesize{\textit{stable}}}
        \put(290,570){\vector(1,0.5){50}}
        \put(750,220){\footnotesize{\textit{stable}}}
        \put(730,230){\vector(-1,0.5){50}}
        \put(750,150){\footnotesize{\textit{unstable}}}
        \put(730,150){\vector(-1,-0.5){50}}
        \put(200,480){\footnotesize{\textit{unstable}}}
        \put(290,460){\vector(1,-0.5){50}}
        \put(800,600){\footnotesize{$L_1$}}
        \put(300,150){\footnotesize{$L_2$}}
        \put(0,400){\rotatebox{270}{\footnotesize{ max$\left(\psi_\ast\right)$}}}
        \put(500,0){\footnotesize{$p$}}
        \end{overpic}}
    \caption{Section of the bifurcation diagram of system (S0) under the assumptions $Ra=4\;10^4$, $\kappa=+\infty$, $L=5$, $\nu=-0.2$ and $\delta=0.5$, for parameter $p$. The skewed southward sinking solution, SS-solution, exists and is stable for $p<\lambda\approx 1.03$. The bifurcation threshold on $p=\lambda$ is associated to a saddle-node bifurcation.\\
    The system is multistable for $p$ close to $\lambda$. In the figure we indicate the salinity-dominated stable solution, SA-solution.\\
    The lines are obtained with the Python library Transiflow \cite{transiflow} through the pseudo-arclength continuation method and uniform resolution grid $(\mathtt{M},\mathtt{N})=(30,60)$.}
    \label{fig:bif_saddle}
\end{figure}
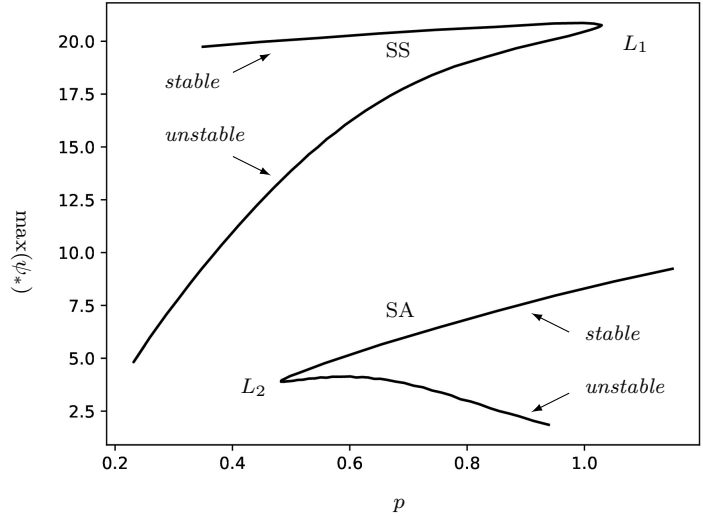

\begin{example}
We consider the model (S1) for parameters $L=5$, $\nu=-0.2$, $\delta=0.5$, $Ra=4\;10^4$. For such values, the surface boundary conditions \eqref{bound3} on $S$ are not symmetric with respect to the mid-axis. We also consider the limit case $\kappa=+\infty$, thus enforcing heterogeneous Dirichlet boundary conditions on $T$. Under such assumption, part of the bifurcation diagram in Figure \ref{fig:bif_saddle} shows a saddle-node bifurcation at $L_1=\lambda\approx1.03$.

\begin{figure}[h!]
    \centering
    \subfloat[Streamfunction $\psi_\text{SS}$]{\begin{overpic}[scale=0.22]{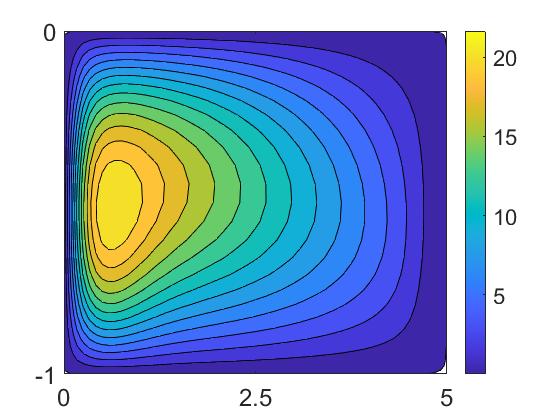}
    \put(-20,380){\rotatebox{270}{\footnotesize{$x_1$}}}
    \put(530,-20){\footnotesize{$x_2$}}
    \end{overpic}}
    \hspace{0mm}
    \subfloat[Vorticity $\omega_\text{SS}$]{\begin{overpic}[scale=0.22]{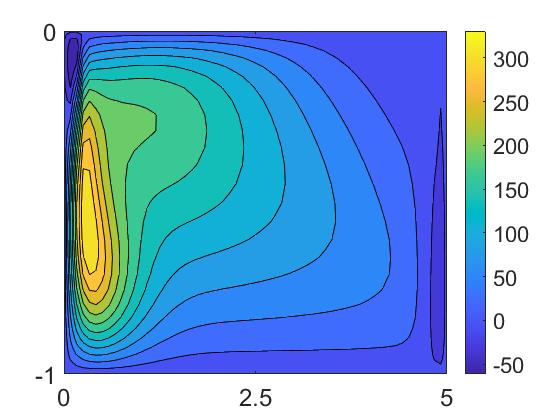}
    \put(-20,380){\rotatebox{270}{\footnotesize{$x_1$}}}
    \put(530,-20){\footnotesize{$x_2$}}
    \end{overpic}}
    \hspace{0mm}
    \subfloat[Temperature $T_\text{SS}$]{\begin{overpic}[scale=0.22]{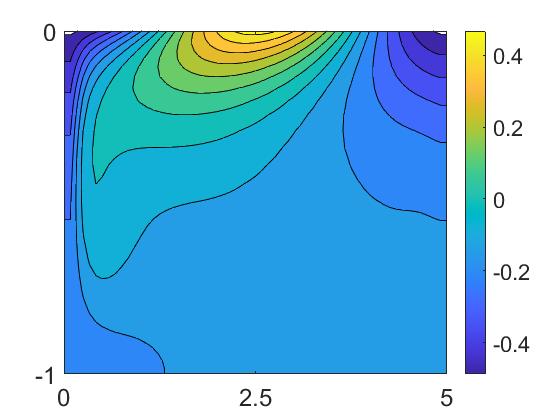}
    \put(-20,380){\rotatebox{270}{\footnotesize{$x_1$}}}
    \put(530,-20){\footnotesize{$x_2$}}
    \end{overpic}}
    \hspace{0mm}
    \subfloat[Salinity $S_\text{SS}$]{\begin{overpic}[scale=0.22]{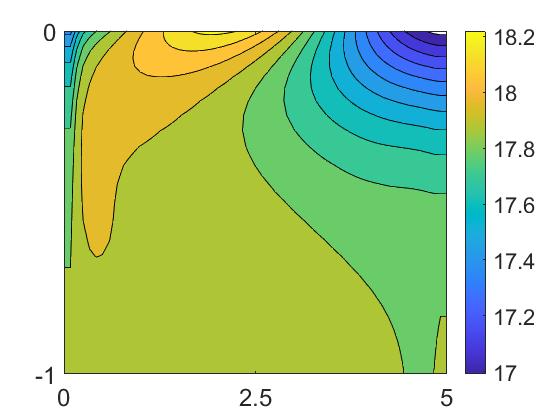}
    \put(-20,380){\rotatebox{270}{\footnotesize{$x_1$}}}
    \put(530,-20){\footnotesize{$x_2$}}
    \end{overpic}}
    
    \vspace{0mm}
    
    \subfloat[Streamfunction $\psi_\text{SA}$]{\begin{overpic}[scale=0.22]{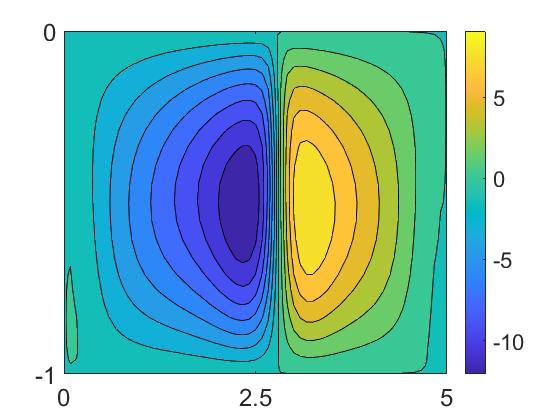}
    \put(-20,380){\rotatebox{270}{\footnotesize{$x_1$}}}
    \put(530,-20){\footnotesize{$x_2$}}
    \end{overpic}}
    \hspace{0mm}
    \subfloat[Vorticity $\omega_\text{SA}$]{\begin{overpic}[scale=0.22]{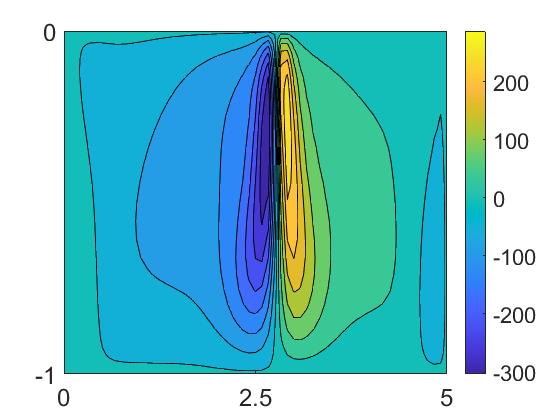}
    \put(-20,380){\rotatebox{270}{\footnotesize{$x_1$}}}
    \put(530,-20){\footnotesize{$x_2$}}
    \end{overpic}}
    \hspace{0mm}
    \subfloat[Temperature $T_\text{SA}$]{\begin{overpic}[scale=0.22]{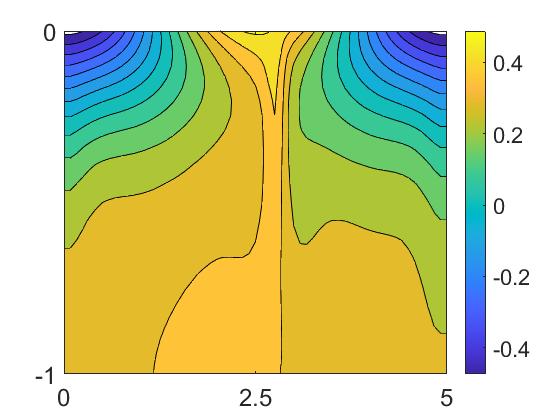}
    \put(-20,380){\rotatebox{270}{\footnotesize{$x_1$}}}
    \put(530,-20){\footnotesize{$x_2$}}
    \end{overpic}}
    \hspace{0mm}
    \subfloat[Salinity $S_\text{SA}$]{\begin{overpic}[scale=0.22]{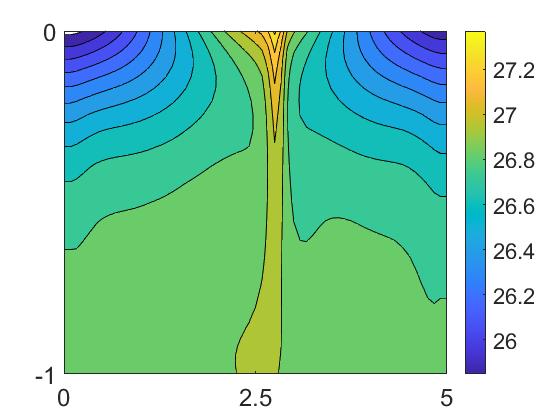}
    \put(-20,380){\rotatebox{270}{\footnotesize{$x_1$}}}
    \put(530,-20){\footnotesize{$x_2$}}
    \end{overpic}}
    \caption{The first row of contour plots presents the simulation of the skewed sinking southward solution $(\psi_\text{SS},\omega_\text{SS},T_\text{SS},S_\text{SS})$ of the system (S0) and assumptions $Ra=4\;10^4$, $\kappa=+\infty$, $L=5$, $\nu=-0.2$ and $\delta=0.5$. We fix $p=1$, therefore less than the saddle-node bifurcation value. The second row displays the salinity-dominated stable solution $(\psi_{\text{SA}},\omega_{\text{SA}},T_{\text{SA}},S_{\text{SA}})$ at $p=1$, under equivalent assumptions.}
    \label{fig:solutions_second}
\end{figure}
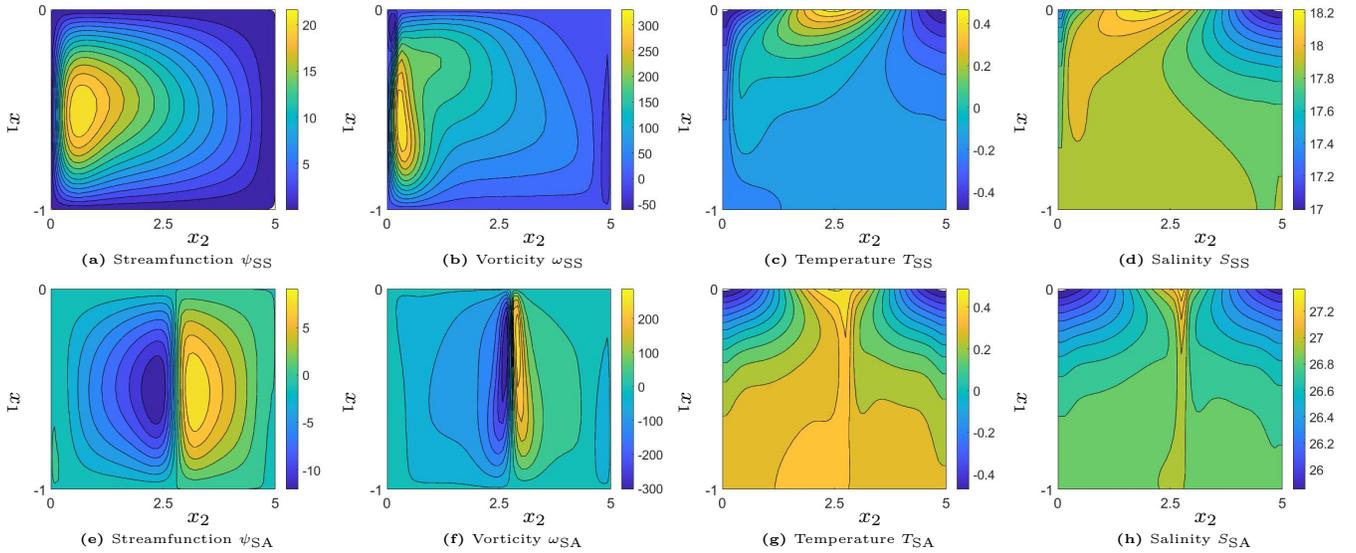

Similarly to the previous example, we study a stable solution of (S0), the skewed sinking southward solution $(\psi_\text{SS},\omega_\text{SS},T_\text{SS},S_\text{SS})$ displayed in the first row of Figure \ref{fig:solutions_second}, as $(\psi_{\ast},\omega_{\ast},T_{\ast},S_{\ast})$. We also observe early-warning signs able to predict the approach of $p$ to $\lambda$ from the behaviour of solutions of the corresponding linearized system (S2). The salinity-dominated solution \cite{dijkstra1997symmetry}, shown in the second row of Figure \ref{fig:solutions_second}, is another stable solution for $p$ in a neighbourhood of $\lambda$. The early-warning signs are, hence, able to predict the jump of a solution of the system (S1) to the proximity of such equilibrium. Yet, the fact that the bifurcation is of saddle-node type implies that, under noise, the jump may happen before the threshold approach. The reliability of the signs is therefore dependent on the noise intensity, which is known to affect the probability of metastable jump before a given time \cite{berglund2023stochastic,bernuzzi2023bifurcations}.

\begin{figure}[h!]
    \centering
    \subfloat{\begin{overpic}[scale=0.5]{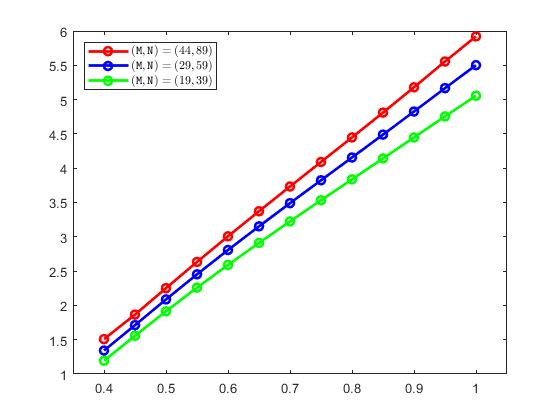}
    \put(500,10){\footnotesize{$p$}}
    \put(10,500){\rotatebox{270}{\footnotesize{$\operatorname{Re}\left(\lambda_2^{(p)}-\lambda_4^{(p)}\right)$}}}
    \end{overpic}}
    \caption{Plot of $\operatorname{Re}\left(\lambda_2^{(p)}-\lambda_4^{(p)}\right)$, for $\lambda_2^{(p)}$ and $\lambda_4^{(p)}$ respectively the second and fourth eigenvalues of $A_0(\psi_\text{SS},\omega_\text{SS},T_\text{SS},S_\text{SS})$ with highest real part for various resolution grid choices. In green the resolution is assumed to be $(\mathtt{M},\mathtt{N})=(19,39)$, in blue to be $(\mathtt{M},\mathtt{N})=(29,59)$ and in red to be $(\mathtt{M},\mathtt{N})=(44,89)$.}
    \label{fig:diff_eigs_second}
\end{figure}

The four eigenvalues of the matrix $A_0(\psi_\text{SS},\omega_\text{SS},T_\text{SS},S_\text{SS})$ with largest real part, for the studied values of $p$, are $\lambda_1^{(p)}=0$, $\lambda_2^{(p)}\neq\overline{\lambda_2^{(p)}}=\overline{\lambda_3^{(p)}}$ and $\lambda_4^{(p)}\in\mathbb{R}$. Figure \ref{fig:diff_eigs_second} displays $\operatorname{Re}\left(\lambda_2^{(p)}-\lambda_4^{(p)} \right)$ for different non-uniform resolution grids, constructed as in the previous example. It is apparent that for $p>0.4$, the difference increases with $p$ and the resolution of the chosen grids. In particular, the difference strays from $0$. This justifies the application of the results of Theorem \ref{thm:general}, Theorem \ref{thm:second}, and Corollary \ref{cor:autocorrelation} on the linearized stochastic system (S2) and the observation of its temporal autocovariance and temporal autocorrelation for long time intervals \cite{KU2}. \\
Figure \ref{fig:rate_bous_second} is obtained from $5$ sample solutions of the approximated linearized system (S2), final time ${t_\text{end}}=10^4$ and the implicit Euler finite difference method, as in the previous example. Figure \ref{fig:10a} and Figure \ref{fig:10b} display in blue respectively $\text{log}_{10}\left(-\frac{1}{\operatorname{Re}(\lambda_2^{(p)})}\right)$ and $\text{log}_{10}\left(-\frac{1}{\operatorname{Re}(\lambda_4^{(p)})}\right)$. Such lines are compared with the mean logarithm of the temporal variance of the solutions projected on the corresponding eigenfunctions of $A_0(\psi_\text{SS},\omega_\text{SS},T_\text{SS},S_\text{SS})^*$, in red. Figure \ref{fig:10a} presents a similar behaviour of the two lines for $p$ close to $\lambda$, aside from the difference of a constant. Such similarity is yet affected by the shape of the noise and of the corresponding eigenfunction, as described in \eqref{eq:general_EWS}. This effect can be observed in Figure \ref{fig:10b}. In such a case, the red line does not indicate divergence of the early-warning sign, similarly to the blue line. 

\begin{figure}[h!]
    \centering
    \subfloat[Comparison of variance and expected scaling law along a sensible direction.]{\begin{overpic}[scale=0.45]{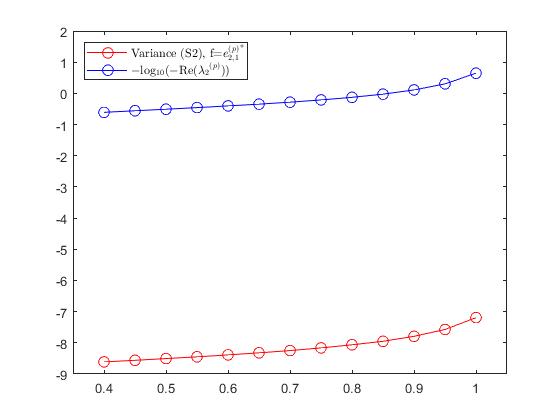}
        \put(500,10){\footnotesize{$p$}}
        \put(10,500){\rotatebox{270}{\footnotesize{$\text{log}_{10}\left(\left\langle f, V_{t_\text{end}} f \right\rangle\right)$}}}
    \end{overpic}
    \label{fig:10a}}
    \hfill
    \subfloat[Comparison of variance and expected scaling law along a not sensible direction.]{\begin{overpic}[scale=0.45]{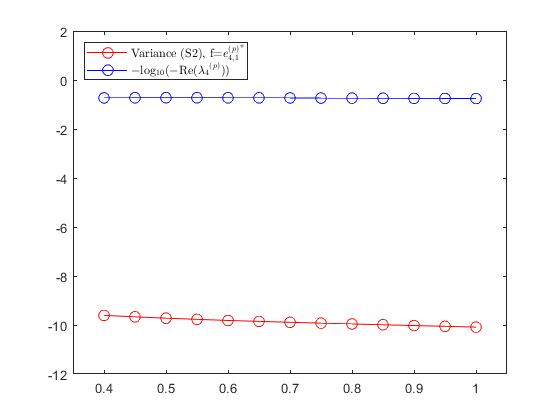}
        \put(500,10){\footnotesize{$p$}}
        \put(10,500){\rotatebox{270}{\footnotesize{$\text{log}_{10}\left(\left\langle f, V_{t_\text{end}} f \right\rangle\right)$}}}
    \end{overpic}
    \label{fig:10b}}
    \caption{Figure \ref{fig:10a} and Figure \ref{fig:10b} show in blue $\text{log}_{10}\left(-\frac{1}{\operatorname{Re}\left(\lambda_2^{(p)}\right)}\right)$ and $\text{log}_{10}\left(-\frac{1}{\operatorname{Re}\left(\lambda_4^{(p)}\right)}\right)$ respectively, for $\lambda_2^{(p)}$ and $\lambda_4^{(p)}$ the second and fourth eigenvalue of $A_0(\psi_\text{SS},\omega_\text{SS},T_\text{SS},S_\text{SS})$ with highest real part. The red lines refer to the mean logarithm of the temporal variance, for time ${t_\text{end}}=10^4$ and $\tau=0$, of solutions of the numerically approximated linearized system (S2), projected on ${e_{2,1}^{(p)}}^*$, in Figure \ref{fig:10a}, and on ${e_{4,1}^{(p)}}^*$, in Figure \ref{fig:10b}. Such lines are obtained from $5$ run samples, which are different from noise realization. 
    }
    \label{fig:rate_bous_second}
\end{figure}
In Figure \ref{fig:eigs_second}, the shape of the simple eigenfunction ${e_{2,1}^{(p)}}^*$ of $A_0(\psi_\text{SS},\omega_\text{SS},T_\text{SS},S_\text{SS})^*$ is displayed for $p=1$. It generates the corresponding generalized eigenspace $E_2(p)^*$, for any observed $p$, and $M_2=1$.\\
For any $p\leq\lambda$, we assume the sequences $\left\{f_1^{(p)}\right\},\left\{f_2^{(p)}\right\}$ continuous in $\mathcal{H}$ and such that the integral of the component on $S$ is $0$, i.e., it is orthogonal to the only generalized eigenfunction, $e_{1,1}^{(p)}$, of $A_0(\psi_\text{SS},\omega_\text{SS},T_\text{SS},S_\text{SS})$ correspondent to $\lambda_1^{(p)}$, as defined in \eqref{eq:e_11}. Under the assumptions of Corollary \ref{cor:autocorrelation}, there exist for any $\delta>0$ the sequences $\left\{g_1^{(p)}\right\}$ and $\left\{g_2^{(p)}\right\}$, generated by a finite number of generalized eigenfunctions of $A_0(\psi_\text{SS},\omega_\text{SS},T_\text{SS},S_\text{SS})^*$, that satisfy \eqref{eq:shadow_cor} and
\begin{equation} \label{eq:EWS_saddle_cov}
    \left\lvert\left\langle g_1^{(p)}, V_\infty^\tau g_2^{(p)}\right\rangle \right\rvert
    = \Theta\left(\operatorname{Re}\left(- \lambda_2^{(p)}\right)^{-1} \right)
\end{equation}
for $p$ that approaches the bifurcation threshold. We can therefore conclude that, although the shape of the function ${e_{2,1}^{(p)}}^*$ is non-trivial, the simple shape of $e_{1,1}^{(p)}$, described in \eqref{eq:e_11}, facilitates the search of other sequences $\left\{g_1^{(p)}\right\}$ and $\left\{g_2^{(p)}\right\}$ that satisfy
\begin{align}
    \begin{alignedat}{2}
    &\left\langle g_1^{(p)}, e_{1,1}^{(p)} \right\rangle = 0 \qquad , \qquad
    &&\left\langle g_2^{(p)}, e_{1,1}^{(p)} \right\rangle = 0 \qquad ,\\
    &\left\langle g_1^{(p)}, e_{2,1}^{(p)} \right\rangle \neq 0 \qquad  \text{and} \qquad
    &&\left\langle g_2^{(p)}, e_{2,1}^{(p)} \right\rangle \neq 0 \qquad ,
    \end{alignedat}
\end{align}
for any $p\leq\lambda$, and, consequently, imply \eqref{eq:EWS_saddle_cov}.

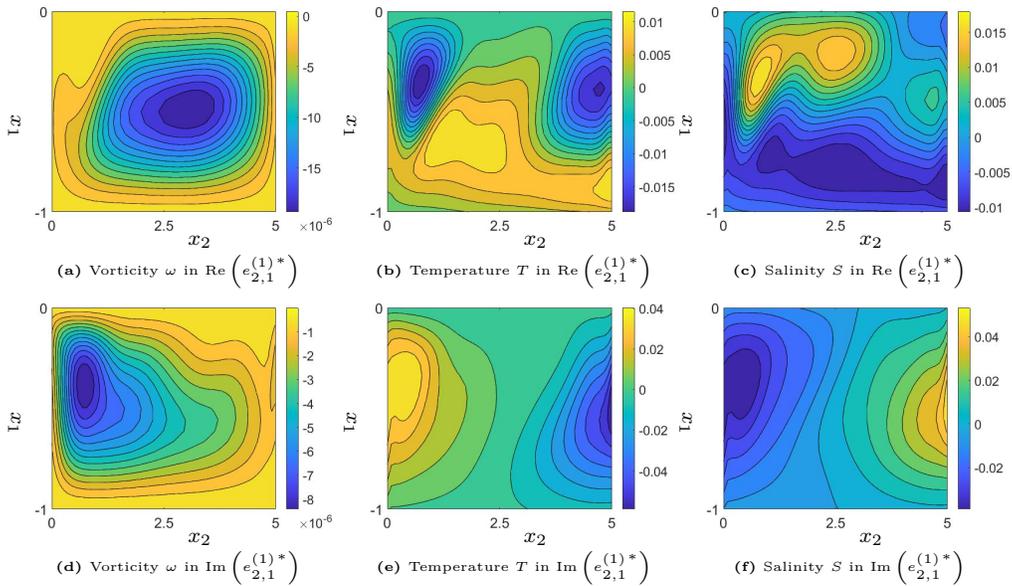
\begin{figure}[h!]
    \centering
    \subfloat[Vorticity $\omega$ in $\operatorname{Re}\left({e_{2,1}^{(1)}}^*\right)$]{\begin{overpic}[scale=0.22]{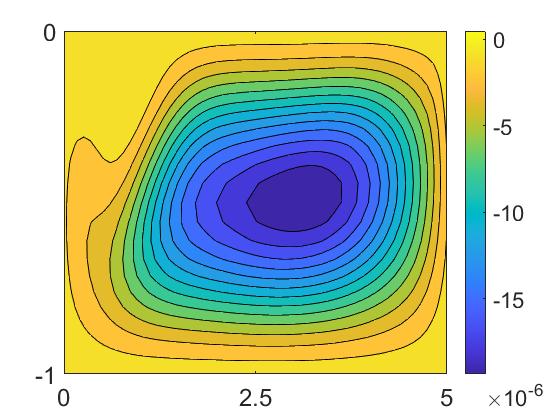}
    \put(-20,380){\rotatebox{270}{\footnotesize{$x_1$}}}
    \put(530,-20){\footnotesize{$x_2$}}
    \end{overpic}}
    \hspace{0mm}
    \subfloat[Temperature $T$ in $\operatorname{Re}\left({e_{2,1}^{(1)}}^*\right)$]{\begin{overpic}[scale=0.22]{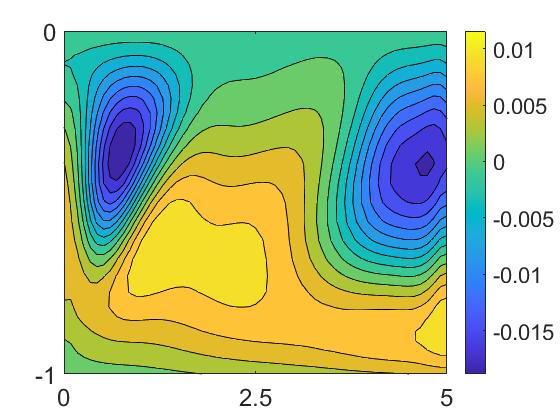}
    \put(-20,380){\rotatebox{270}{\footnotesize{$x_1$}}}
    \put(530,-20){\footnotesize{$x_2$}}
    \end{overpic}}
    \hspace{0mm}
    \subfloat[Salinity $S$ in $\operatorname{Re}\left({e_{2,1}^{(1)}}^*\right)$]{\begin{overpic}[scale=0.22]{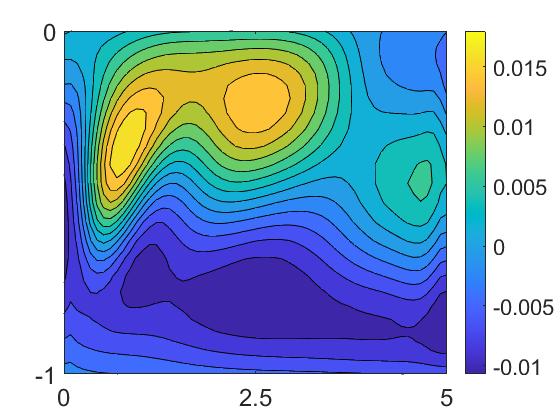}
    \put(-20,380){\rotatebox{270}{\footnotesize{$x_1$}}}
    \put(530,-20){\footnotesize{$x_2$}}
    \end{overpic}}
    
    \vspace{0mm}
    
    \subfloat[Vorticity $\omega$ in $\operatorname{Im}\left({e_{2,1}^{(1)}}^*\right)$]{\begin{overpic}[scale=0.22]{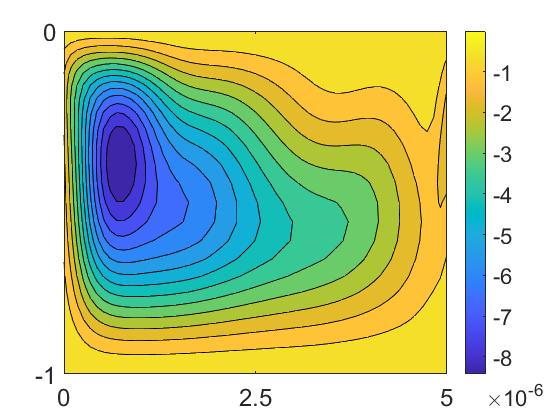}
    \put(-20,380){\rotatebox{270}{\footnotesize{$x_1$}}}
    \put(530,-20){\footnotesize{$x_2$}}
    \end{overpic}}
    \hspace{0mm}
    \subfloat[Temperature $T$ in $\operatorname{Im}\left({e_{2,1}^{(1)}}^*\right)$]{\begin{overpic}[scale=0.22]{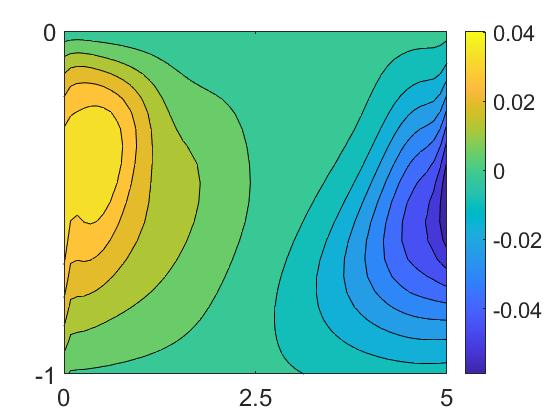}
    \put(-20,380){\rotatebox{270}{\footnotesize{$x_1$}}}
    \put(530,-20){\footnotesize{$x_2$}}
    \end{overpic}}
    \hspace{0mm}
    \subfloat[Salinity $S$ in $\operatorname{Im}\left({e_{2,1}^{(1)}}^*\right)$]{\begin{overpic}[scale=0.22]{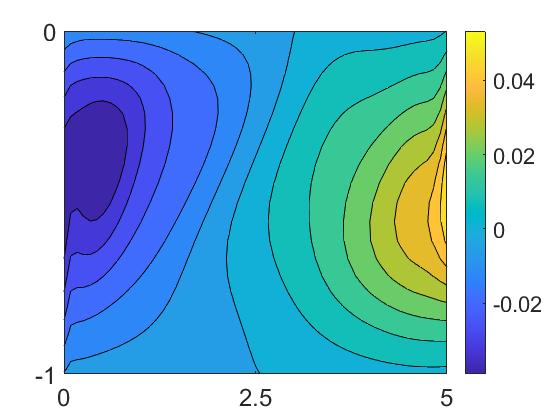}
    \put(-20,380){\rotatebox{270}{\footnotesize{$x_1$}}}
    \put(530,-20){\footnotesize{$x_2$}}
    \end{overpic}}
    \caption{The two rows of contour plots display respectively the real and imaginary parts of the components of the numerical approximations of ${e_{2,1}^{(1)}}^*$. Such function is the unique generalized eigenfunction of $A_0(\psi_\text{SS},\omega_\text{SS},T_\text{SS},S_\text{SS})^*$ associated to $\overline{\lambda_2^{(p)}}$. We label its real and imaginary parts as $\operatorname{Re}\left({e_{2,1}^{(1)}}^*\right)$ and $\operatorname{Im}\left({e_{2,1}^{(1)}}^*\right)$.}
    \label{fig:eigs_second}
\end{figure} 

In Figure \ref{fig:autocorr}, the mean temporal autocorrelation of $5$ solutions of the linearized system (S2) in time $t_{\text{end}}=10^4$ is studied in relation to the lag $0\leq\tau\leq10$ and fixed parameter values $p$. For $i\in\{2,4\}$, the plots refer to the absolute values of the numerical approximation of $\hat{V}_{t_{\text{end}}}^\tau \left({e_{i,1}^{(p)}}^*, {e_{i,1}^{(p)}}^*\right)$, in solid lines, and the absolute values of the function $e^{\overline{\lambda_i^{(p)}}\tau}$, denoted by circles. Corollary \ref{cor:autocorrelation} expects the $L^2$-norm of the differences of $\hat{V}_\infty^\tau \left({e_{i,1}^{(p)}}^*, {e_{i,1}^{(p)}}^*\right)$ and $e^{\overline{\lambda_i^{(p)}}\tau}$ on the interval $\tau\in[0,10]$ to be $0$ for any $0.4\leq p\leq1$. Such numerical errors are of order $10^{-3}$ for all values $i\in\{2,4\}$ and $p\in\{0.4,0.6,0.8,1\}$ reported in the figure.

\begin{figure}[h!]

    \centering

    \subfloat[Comparison of autocorrelation and expected exponential behaviour along a sensible direction.]{\begin{overpic}[scale=0.45]{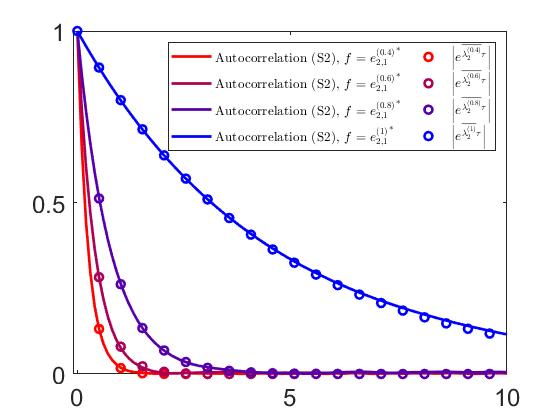}
        \put(590,20){\footnotesize{$\tau$}}
        \put(10,450){\rotatebox{270}{\footnotesize{$\left\lvert\hat{V}_{t_\text{end}}^{\tau}\left(f,f \right)\right\rvert$}}}
    \end{overpic}
    \label{fig:12a}}
    \hfill
    \subfloat[Comparison of autocorrelation and expected exponential behaviour along a not sensible direction.]{\begin{overpic}[scale=0.45]{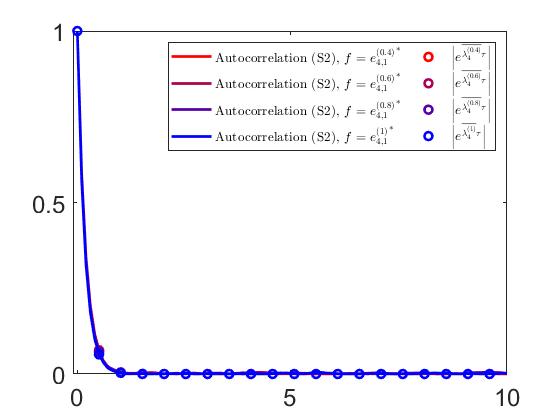}
        \put(590,20){\footnotesize{$\tau$}}
        \put(10,450){\rotatebox{270}{\footnotesize{$\left\lvert\hat{V}_{t_\text{end}}^{\tau}\left(f,f \right)\right\rvert$}}}
    \end{overpic}
    \label{fig:12b}}
    
    \caption{
    The plots are obtained from $5$ solutions, differing from noise samples, of the numerically approximated linearized system (S2) until ${t_\text{end}}=10^4$. The solid lines indicate the absolute value of the mean temporal autocorrelation of the solutions projected on ${e_{i,1}^{(p)}}^*$. In contrast, the circles display the absolute value of $e^{\overline{\lambda_i^{(p)}} \tau}$. The functions are presented with the fixed parameter $p$ and under the dependence of the lag time $\tau$.}
    \label{fig:autocorr}
\end{figure}

\end{example}

\section*{Conclusion and Outlook}

In many applications, stochastic models show solutions that are characterized by sudden critical transitions. The prediction of such events holds particular importance from a theoretical perspective and for the impact on the corresponding real-life phenomena. In this work, we focus on systems affected on the boundary by white noise. We construct two early-warning signs as the time-asymptotic autocovariance and autocorrelation of the solution of the corresponding linearized fast system. We also provide a description of the directions in the square-integrable functions space along which such objects can predict the crossing of a bifurcation threshold, and discuss the observation of the signs in several examples. In detail, under the referred assumptions, we prove that the divergent time-asymptotic autocovariance, along a general choice of functions, is able to forecast the bifurcation in the linearized fast system. Conversely, the time-asymptotic autocorrelation provides a more precise prediction if observed along a restrictive set of direction functions.\\
Particular focus is given to applying the tools to a Boussinesq ocean model, such as examining the projection of the solution on certain eigendirections in the proximity of bifurcation thresholds and the rates adopted by the observables. The study is carried out on a supercritical pitchfork bifurcation and a saddle-node bifurcation. The role of multistability and linearization in the warning of the bifurcation event is examined. The simulations cross-validate the analytic results and provide further insight into the correct practical use of the tools. The results expand the theory of early-warning signs for SPDEs \cite{bernuzzi2023early,bernuzzi2023bifurcations, KU2, KU}. The construction of similar tools under different types of noise may present interesting mathematical implications and increase the applicability of the field \cite{layritz2023early}.

\appendix

\section{Appendix: Properties of the symmetric Boussinesq model.} \label{app:math_appendix}

We fix $\nu=0$, $Ra,L,\kappa>0$ and $p,\delta\in\mathbb{R}$. This appendix presents proof of the fact that for any steady solution of the system (S0), there exists another equilibrium that shares its stability. This is meant in the sense that the operator $A_0$, defined in \eqref{eq:A0_1} and \eqref{eq:A0_2}, applied on each corresponding solution of (S0), has the same spectrum.\\
We introduce the idempotent reflection operator 
\begin{equation*}
  \Upsilon:L^2([-H,0]\times[0,L])\mapsto L^2([-H,0]\times[0,L])\;,
\end{equation*}
such that
\begin{align*}
    &\Upsilon f(x_1,x_2)=f(x_1,L-x_2) \;,
\end{align*}
for any $f\in L^2([-H,0]\times[0,L])$ and $(x_1,x_2)\in[-H,0]\times[0,L]$. For a steady solution $(\psi_1,\omega_1,T_1,S_1)$ of the original system (S0), we define the new proposed solution $(\psi_2,\omega_2,T_2,S_2)$ as follows:
\begin{align*}
    \psi_2=-\Upsilon\psi_1\;,
    \quad \omega_2=-\Upsilon\omega_1\;,
    \quad S_2=\Upsilon S_1\quad \text{and} 
    \quad T_2=\Upsilon T_1 \;.
\end{align*}
Then the following hold:
\begin{align*}
    &\lozenge\qquad\qquad -\frac{\partial \psi_2}{\partial x_1} \frac{\partial \omega_2}{\partial x_2} +\frac{\partial \psi_2}{\partial x_2} \frac{\partial \omega_2}{\partial x_1}+ Pr \Delta \omega_2 + Pr\;Ra \left( \frac{\partial T_2}{\partial x_2} - \frac{\partial S_2}{\partial x_2} \right) \\
    &\;\;\;\qquad\qquad=-\Upsilon\left(-\frac{\partial \psi_1}{\partial x_1} \frac{\partial \omega_1}{\partial x_2} +\frac{\partial \psi_1}{\partial x_2} \frac{\partial \omega_1}{\partial x_1}+ Pr \Delta \omega_1 + Pr\;Ra \left( \frac{\partial T_1}{\partial x_2} - \frac{\partial S_1}{\partial x_2} \right)\right)=0\;,\\
    &\lozenge\qquad\qquad \omega_2 =-\Upsilon \omega_1 = \Upsilon \Delta \psi_1 = - \Upsilon \Delta \Upsilon \psi_2 = -\Delta \psi_2 \;,\\
    &\lozenge\qquad\qquad-\frac{\partial \psi_2}{\partial x_1} \frac{\partial T_2}{\partial x_2} +\frac{\partial \psi_2}{\partial x_2} \frac{\partial T_2}{\partial x_1}+ \Delta T_2 
    = \Upsilon\left(-\frac{\partial \psi_1}{\partial x_1} \frac{\partial T_1}{\partial x_2} +\frac{\partial \psi_1}{\partial x_2} \frac{\partial T_1}{\partial x_1}+ \Delta T_1 \right)=0 \;,\\
    &\lozenge\qquad\qquad -\frac{\partial \psi_2}{\partial x_1} \frac{\partial S_2}{\partial x_2} +\frac{\partial \psi_2}{\partial x_2} \frac{\partial S_2}{\partial x_1}+ Le^{-1} \Delta S_2 
    = \Upsilon\left(-\frac{\partial \psi_1}{\partial x_1} \frac{\partial S_1}{\partial x_2} +\frac{\partial \psi_1}{\partial x_2} \frac{\partial S_1}{\partial x_1}+ Le^{-1} \Delta S_1 \right)=0 \;.
\end{align*}
Such equations are given by $\Upsilon \Delta \Upsilon =\Delta$ under homogeneous Dirichlet boundary conditions. The boundary conditions are satisfied since $Q_S$ and $T_S$ are symmetric functions in respect to $x_2=\frac{L}{2}$.\\
The stability of $(\psi_2,\omega_2,T_2,S_2)$ is studied by showing that $A_0(\psi_2,\omega_2,T_2,S_2)$ shares the spectrum of $A_0(\psi_1,\omega_1,T_1,S_1)$. This is given by the fact that 
\begin{align*}
A_\text{S}(\psi_1,\omega_1,T_1,S_1)=
\begin{pmatrix}
-\Upsilon & 0 & 0 \\ 0 & \Upsilon & 0 \\ 0 & 0 & \Upsilon
\end{pmatrix}
A_\text{S}(\psi_2,\omega_2,T_2,S_2)
\begin{pmatrix}
-\Upsilon & 0 & 0 \\ 0 & \Upsilon & 0 \\ 0 & 0 & \Upsilon
\end{pmatrix}
\end{align*}
and by noting that the operator $\begin{pmatrix}
-\Upsilon & 0 & 0 \\ 0 & \Upsilon & 0 \\ 0 & 0 & \Upsilon
\end{pmatrix}$ is unitary, or a change of basis.

\newpage
\printbibliography

\end{document}